\documentclass{amsart}
\usepackage{amsfonts,amssymb,amscd,amsmath,latexsym,amsbsy,mathrsfs,amsxtra}
\usepackage{mathtools,tikz-cd}
\usetikzlibrary{cd}
\usetikzlibrary{arrows}
\usetikzlibrary{matrix}
\usepackage{graphicx,ctable,booktabs}
\usepackage[all]{xy}
\usetikzlibrary{arrows.meta}
\numberwithin{equation}{section}            
\theoremstyle{plain}
\newtheorem{thm}{Theorem}[section]

\newtheorem{prop}[thm]{Proposition}

\newtheorem{defi}[thm]{Definition}
\newtheorem{lem}[thm]{Lemma}
\newtheorem{cor}[thm]{Corollary}
\newtheorem{eg}[thm]{{Example}}

\theoremstyle{remark}
\newtheorem{rema}[thm]{Remark}

\usepackage{enumerate}
\usepackage{comment}

\newcommand{\ad}{{\mbox{\upshape{ad}}}}

\newcommand{\barM}{\overline{\phantom{m}}^\cM}

\newcommand{\bc}{{\mathbf{c}}}

\newcommand{\bs}{{\bf{s}}}
\newcommand{\C}{{\mathbb C}}

\newcommand{\N}{{\mathbb N}}
\newcommand{\cA}{{\mathcal A}}

\newcommand{\cB}{{\mathcal B}}
\newcommand{\cC}{{\mathcal C}}

\newcommand{\cF}{{\mathcal F}}

\newcommand{\cH}{{\mathcal H}}

\newcommand{\cL}{{\mathcal L}}
\newcommand{\cM}{{\mathcal M}}

\newcommand{\cP}{{\mathcal P}}
\newcommand{\cR}{{\mathcal R}}
\newcommand{\cS}{{\mathcal S}}

\newcommand{\cZ}{{\mathcal Z}}

\newcommand{\Du}{u}
\newcommand{\Dc}{C}

\newcommand{\End}{\mathrm{End}}
\newcommand{\Etil}{\widetilde{E}}

\newcommand{\field}{{\mathbb K}}

\newcommand{\gfrak}{{\mathfrak g}}

\newcommand{\gr}{{\mathrm{gr}}}

\newcommand{\Hom}{{\mathrm{Hom}}}

\newcommand{\id}{{\mathrm{id}}}

\newcommand{\kow}{{\varDelta}}

\newcommand{\ot}{\otimes}

\newcommand{\stirling}[2]{\genfrac{[}{]}{0pt}{}{#1}{#2}}

\newcommand{\Uq}{U}

\newcommand{\uqg}{{U_q(\mathfrak{g})}}

\newcommand{\Upoly}{U^{\mathrm{poly}}}

\newcommand{\Z}{{\mathbb Z}}

\catcode`,\active

\catcode`\,12

\begin{document}
\title[]
{Defining relations of quantum symmetric pair coideal subalgebras}
\author[Stefan Kolb]{Stefan Kolb}
\address{School of Mathematics, Statistics and Physics,
Newcastle University, Newcastle upon Tyne NE1 7RU, United Kingdom}
\email{stefan.kolb@newcastle.ac.uk}
\author[Milen Yakimov]{Milen Yakimov}
\address{
Department of Mathematics, Northeastern University, Boston, MA 02115, U.S.A.
}
\email{m.yakimov@northeastern.edu}
\thanks{The research of M.Y. was supported by NSF grant DMS-1901830 and Bulgarian Science Fund grant DN02/05.}

\keywords{Quantum symmetric pairs, star products on $\N$-graded noncomutative algebras, deformed Chebyshev poslynomials of the second kind, bivariate continuous $q$-Hermite polynomials}

\subjclass[2010]{Primary: 17B37, Secondary: 53C35, 16T05, 17B67}
\begin{abstract}
We explicitly determine the defining relations of all 
quantum symmetric pair coideal subalgebras of quantized enveloping algebras of Kac--Moody type.
Our methods are based on star products on noncommutative $\N$-graded algebras. The resulting defining relations are expressed in terms of continuous $q$-Hermite polynomials and a new family of deformed Chebyshev polynomials.
\end{abstract}
\maketitle

\section{Introduction}
\subsection{Background} Let $\gfrak$ be a symmetrizable Kac-Moody algebra and $\theta:\gfrak\rightarrow \gfrak$ an involutive Lie algebra automorphism. In the theory of quantum symmetric pairs one considers quantum group analogs of the fixed Lie subalgebra $\gfrak^\theta=\{x\in \gfrak\,|\,\theta(x)=x\}$. The main objects of investigation are certain subalgebras $\cB_\bc$ of the quantized enveloping algebra $\uqg$, which specialize to $U(\gfrak^\theta)$ in a suitable limit $q\rightarrow 1$. The subalgebras $\cB_\bc$ satisfy the (right) coideal property $\kow(\cB_\bc)\subset \cB_\bc\ot \uqg$ where $\kow$ denotes the coproduct of $\uqg$, and we refer to them as quantum symmetric pair coideal subalgebras.

For $\gfrak$ of finite type a comprehensive theory of quantum symmetric pairs was developed by G.~Letzter in \cite{a-Letzter99a}. This theory was extended to the Kac-Moody case in \cite{a-Kolb14} for involutive automorphisms of the second kind of $\gfrak$. Such involutions are determined by Satake diagrams $(X,\tau)$ where $X$ is a subset of the nodes $I$ of the underlying Dynkin diagram, and $\tau:I\rightarrow I$ is an involutive diagram automorphism. The subset $X\subset I$ has to be of finite type and the pair $(X,\tau)$ has to satisfy the compatibility conditions given in \cite[Definition 2.3]{a-Kolb14}. It was observed by V.~Regelskis and B.~Vlaar in \cite{a-RV20} that these conditions can be slightly weakened and that the theory extends to a setting of generalized Satake diagrams $(X,\tau)$, see Section \ref{sec:partial-parabolic}.

The theory of quantum symmetric pairs has been evolving rapidly over the past decade. In \cite{a-BaoWang18a} H.~Bao and W.~Wang initiated a program to extend Lusztig's theory of canonical bases to the setting of quantum symmetric pairs. This program fed into a general construction of a universal $K$-matrix in \cite{a-BalaKolb19} which is the quantum symmetric pair analog of the universal $R$-matrix for quantum groups. A centerpiece in both constructions was a notion of bar involution for quantum symmetric pairs which appeared independently in \cite{a-ES18} and \cite{a-BaoWang18a}. These developments led to a flurry of activity aiming to extend many quantum group related constructions to the setting of quantum symmetric pairs.
\subsection{The problem}\label{sec:problem}
One of the outstanding problems in the theory of quantum symmetric pairs is to give an explicit, conceptual and simple description of the algebra $\cB_\bc$ in terms of generators and relations for all generalized Satake diagrams $(X,\tau)$. In the present paper we solve this problem completely. To describe previous work and to formulate our results, we recall the definition of the quantum symmetric pair coideal subalgebra $\cB_\bc$ in terms of the pair $(X,\tau)$.

Let $U=U_q(\gfrak')$ be the quantized enveloping algebra of the derived Lie algebra $\gfrak'=[\gfrak,\gfrak]$ with standard generators $E_i, F_i, K_i^{\pm 1}$ for $i\in I$, defined over the field $\field=k(q)$ of rational functions in an indeterminate $q$. Let $\cM_X\subset U$ be the subalgebra generated by $E_i, F_i, K_i^{\pm 1}$ for $i\in X$ and let $U^0_\Theta$ be the subalgebra generated by $K_j, K_i K_{\tau(i)}^{-1}$ for $j\in X, i\in I\setminus X$. Let $W$ be the Weyl group of $\gfrak$ and let $w_X$ be the longest element in the parabolic subgroup of $W$ for the subset $X$. We write $T_{w_X}$ to denote the corresponding Lusztig automorphism. By construction $U^0_\Theta$ is the group algebra of the subgroup $Q^\Theta=\{\beta\in Q\,|\,\beta=-w_X\tau(\beta)\}$ of the root lattice $Q$. Let $K_\beta$ for $\beta\in Q^\Theta$ be the corresponding basis element of $U^0_\Theta$. For each $i\in I\setminus X$ define
\begin{align}\label{eq:Bi-def-intro}
  B_i=F_i- c_iT_{w_X}(E_{\tau(i)})K_i^{-1}
\end{align}
where $c_i\in \field^\ast$. For a good theory, we need to assume that the parameters $\bc=(c_i)_{i\in I\setminus X}$ belong to the set
\begin{align}\label{eq:C-def-intro}
  \cC=\{\bc=(c_i)_{i\in I\setminus X}\in (\field^\times)^{I\setminus X}\,|\, c_i=c_{\tau(i)} \mbox{ if $(\alpha_i,w_X(\alpha_{\tau(i)}))=0$}\}.
\end{align}  
By definition, the quantum symmetric pair coideal subalgebra $\cB_\bc$ for $\bc=(c_i)_{i\in I\setminus X}\in \cC$ is the subalgebra of $U$ generated by $\cM_X$, $U^0_\Theta$ and the elements $B_i$ given by \eqref{eq:Bi-def-intro}. For $j\in X$ we also write $B_j=F_j$.

The algebra $\cB_\bc$ has a filtration $\cF$ defined by a degree function given by
\begin{equation}
\begin{aligned}\label{eq:filt-def}
  \deg(B_i)&=1 \qquad \mbox{for $i\in I\setminus X$},\\
  \deg(h)&=0 \qquad \mbox{for $h\in \cM_X U^0_\Theta$.}
\end{aligned}
\end{equation}
For $i,j\in I$ let $S_{ij}(x,y)\in \field[x,y]$ be the quantum Serre polynomial given by \eqref{eq:qSerre}. Let $(a_{ij})_{i,j\in I}$ be the generalized Cartan matrix of $\gfrak$. By definition one has $S_{ij}(B_i,B_j)\in \cF_{\deg(i,j)}(\cB_\bc)$ where 
\begin{align*}
   \deg(i,j)=\begin{cases}
              2-a_{ij} & \mbox{if $i,j\in I\setminus X$,}\\
              1-a_{ij} & \mbox{if $i\in I\setminus X$, $j\in X$,}\\
              1 & \mbox{if $i\in X$, $j\in I\setminus X$,}\\
              0& \mbox{if $i,j\in X$.}\\
            \end{cases}
\end{align*}
A finer analysis implies that there exist elements $C_{ij}(\bc)\in \cF_{\deg(i,j)-1}(\cB_\bc)$ such that $S_{ij}(B_i,B_j)=C_{ij}(\bc)$ for all $i,j\in I$. Let $\cM_X^+$ denote the subalgebra of $U$ generated by all $E_j$ for $j\in X$. The following theorem was proved in \cite{MSRI-Letzter} for $\gfrak$ of finite type and was extended to the Kac-Moody case in \cite{a-Kolb14}.
\begin{thm}\label{thm:initial-ref}
\cite[Theorem 7.4]{MSRI-Letzter}, \cite[Theorem 7.1]{a-Kolb14}
  Let $\bc\in \cC$. The algebra $\cB_\bc$ is generated over $\cM_X^+ U^0_\Theta$ by the elements $B_i$ for $i\in I$ subject to the following relations
  \begin{align}
    K_\beta B_i &=q^{-(\beta, \alpha_i)} B_i K_\beta &&\mbox{for all $\beta\in Q^\Theta$, $i\in I$}\nonumber\\
    E_iB_j-B_jE_i&=\delta_{ij}\frac{K_i-K_i^{-1}}{q_i-q_i^{-1}} &&\mbox{for all $i\in X,j\in I$,}\nonumber\\
  S_{ij}(B_i,B_j)&=C_{ij}(\bc) && \mbox{for all $i,j\in I$, $i\neq j$.}\label{eq:S=C-intro}
  \end{align}
\end{thm}  
We call the relations \eqref{eq:S=C-intro} the quantum Serre relations for $\cB_\bc$. In order to obtain defining relations for $\cB_\bc$, it remains to determine these quantum Serre relations explicitly.
\subsection{Previous results}\label{sec:prev}
It follows from \cite[Lemma 5.11, Theorem 7.3]{a-Kolb14} that $C_{ij}(\bc)=0$ if $i\in X$ or $\tau(i)\notin\{i,j\}$. Hence it remains to determine the relations \eqref{eq:S=C-intro} explicitly in the following three cases:
\begin{enumerate}
  \item[(I)] $\tau(i)=i$ and $i,j\in I\setminus X$;
  \item[(II)]$\tau(i)=i$ and $i\in I\setminus X$, $j\in X$;
  \item[(III)] $\tau(i)=j$ and $i,j\in I\setminus X$. 
\end{enumerate}
For all Satake diagrams of finite type, the relations \eqref{eq:S=C-intro} were determined explicitly in \cite[Theorem 7.1]{a-Letzter03} by a subtle method involving the coproduct $\kow$ of $U$. Letzter's method was extended to the Kac-Moody case in \cite[Section 7]{a-Kolb14} and was used to determine the relations \eqref{eq:S=C-intro} in the case $|a_{ij}|\le 2$. A general formula for $C_{ij}(\bc)$ in case (III) was obtained in \cite[Theorem 3.6]{a-BalaKolb15}, again using Letzter's coproduct method. In \cite[Theorem 3.9]{a-BalaKolb15} it was observed that cases (I) and (II) should allow a uniform treatment, but explicit formulas were still elusive for $|a_{ij}|>3$.

Following \cite{a-CLW20} we call a quantum symmetric pair quasi-split if $X=\emptyset$. Case (II) does not appear in the quasi-split setting. Explicit formulas for the quantum Serre relations \eqref{eq:S=C-intro} in the quasi-split case (I) were given in \cite{a-CLW20} in terms of so-called $\imath$divided powers for $\cB_\bc$. The methods in \cite{a-CLW20} are calculational but do not involve Letzter's coproduct method. 

Using Letzter's coproduct method, H.~de Clercq was able to derive expressions for the quantum Serre relations \eqref{eq:S=C-intro} in general. Cases (I) and (II) are treated in \cite[Theorems 3.13, 3.19]{a-dC19p} by combinatorially involved and unwieldy formulas. Nonetheless, in case (I), de Clercq was able to use her formulas to extend the quantum Serre relations from \cite{a-CLW20} from the quasi-split case to the case of general pairs $(X,\tau)$, see \cite[Theorem 4.7]{a-dC19p}. This extension of results was also performed in \cite{a-CLW21} in a more general setting of higher Serre relations. A conceptual and compact expression for the quantum Serre relations in case (II), however, remained to be found.

In a new approach \cite{a-CKY20},
W.~R.~Casper and the authors also found explicit, conceptual formulas for the quantum Serre relations \eqref{eq:S=C-intro} in the quasi-split case. 
The resulting relations are expressed in terms of continuous $q$-Hermite polynomials. The formulas are obtained using the star product interpretation 
of quantum symmetric pairs from \cite{a-KY20}. Continuous $q$-Hermite polynomials differ from the $\imath$divided powers in \cite{a-CLW20}, 
but the resulting descriptions of the quantum Serre relations bear many similarities. The method in \cite{a-CKY20} involves minimal calculations.

In the present paper we further develop the star product interpretation of quantum symmetric pairs to give a conceptual and uniform 
treatment of the quantum Serre relations \eqref{eq:S=C-intro} in all of the there cases (I), (II), (III) for arbitrary generalized Satake diagrams. 
This, in particular, settles the hardest and subtle case (II). We give a complete description of the defining relations of $\cB_\bc$ by closed formulas,
which is independent of any of the previous approaches and is based on minimal calculations with star products.
\subsection{Statement of results}\label{sec:statement}
To formulate our results we need three families of $\cM_X^+$-valued orthogonal polynomials. For $i\in I\setminus X$ define
\begin{align}\label{eq:cZi}
\cZ_i=q_i c_i \partial^R_{\tau(i)}(T_{w_X}(E_{\tau(i)}))
\end{align}
where $\partial^R_{\tau(i)}$ is the Lusztig-Kashiwara skew-derivation for $\Uq$, see Section \ref{sec:setting}. The element $\cZ_i$ belongs to $\cM_X^+$ and hence commutes with $B_\ell$ for all $\ell\in I\setminus X$.

\medskip

{\bf A)} The \textit{rescaled univariate $q$-Hermite polynomials} $w_m(x;q_i^2)\in \cM_X^+[x]$ are defined for $m\in \N$ by the initial conditions $w_0(x;q_i^2)=1$, $w_1(x;q_i^2)=x$ and the recursion
\begin{align*}
  w_{m+1}(x;q_i^2)= x w_m(x;q_i^2)- \frac{1-q_i^{2m}}{(q_i-q_i^{-1})^2} \cZ_i w_{m-1}(x;q_i^2).
\end{align*}
The polynomials $w_m(x;q_i^2)$ are rescaled versions of the univariate continuous $q$-Hermite polynomials defined for example in \cite[14.26]{b-KLS10}, see Sections \ref{sec:biv-q-Herm}, \ref{sec:o-p} for details.

\medskip

{\bf B)} The \textit{rescaled bivariate continuous $q$-Hermite polynomials} $w_{m,n}(x,y;q_i^2,q_i^{a_{ij}})\in \cM_X^+[x,y]$ for $m,n\in \N$ are defined by the initial condition $w_{0,n}(x,y;q_i^2,q_i^{a_{ij}})=w_n(y;q_i^2)$ for all $n\in \N$ and the recursion
\begin{align*}
  x w_{m,n}(x,y;q_i^2,q_i^{a_{ij}})= w_{m+1,n}(x,y;q_i^2,&\, q_i^{a_{ij}}) +\frac{1-q_i^{2m}}{(q_i-q_i^{-1})^2} \cZ_i w_{m-1,n}(x,y; q_i^2,q_i^{a_{ij}})\\
                     +&  \frac{q_i^{2m}(1-q_i^{2n})}{(q_i-q_i^{-1})^2}q_i^{a_{ij}}\cZ_i w_{m,n-1}(x,y;q_i^2,q_i^{a_{ij}}). \nonumber
\end{align*}
The polynomials $w_{m,n}(x,y;q_i^2,q_i^{a_{ij}})$ are rescaled versions of the bivariate continuous $q$-Hermite polynomials $H_{m,n}(x,y;q,r)$ which were first introduced and studied in \cite{a-CKY20}, see Section \ref{sec:biv-q-Herm} for more details.

\medskip

{\bf C)} The \textit{rescaled deformed Chebyshev polynomials} $\Du_n(x;q_i^2,q_i^{2a_{ij}})\in \cM_X^+[x]$ are defined for $n\in \N$ by $\Du_0(x;q_i^2,q_i^{2a_{ij}})=1$, $\Du_1(x;q_i^2,q_i^{2a_{ij}})=2x$ and the recursion
\begin{align*}
  \Du_{n+1}(x;q_i^2,q_i^{2a_{ij}})= 2x\Du_n(x;q_i^2,q_i^{2a_{ij}}) - \frac{q_i^{-2a_{ij}}-q_i^{2(n+1)}}{(1-q_i^{2(n+1)})(q_i-q_i^{-1})^2}\cZ_i \Du_{n-1}(x;q_i^2,q_i^{2a_{ij}}).
\end{align*}
The polynomials $\Du_n(x;q_i^2,q_i^{2a_{ij}})$ are obtained by rescaling from a family of univariate orthogonal polynomials $\Dc_n(x;q,r)\in \field[x]$ for $n\in \N$, $r\in \field$, which we call  
{\em{deformed Chebyshev polynomials of the second kind}}. The polynomials $\Dc_n(x;q,r)$ are defined by the initial conditions $\Dc_0(x; q, r) =1$, $\Dc_1(x; q, r) =2x$ and the recursion 
\[
\Dc_{n+1}(x; q, r) = 2x \Dc_{n}(x; q, r) - \frac{r^{-1}-q^{n+1}}{1 - q^{n+1}} \Dc_{n-1}(x; q, r).
\]
In the case $r=1$ we recover the classical {\em{Chebyshev polynomials of the second kind}}.
We investigate this new sequence of orthogonal polynomials in Section \ref{sec:Chebyshev}, obtaining two generating functions for them, the first of which is
\[
\sum_{n=0}^\infty \Dc_n(x;q,r) \frac{s^n}{1-q^{n+2}} =\frac{1}{1-2xs+r^{-1}s^2} {}_3\phi_2\left(\begin{array}{c} q,x_1s,x_2s \\qa_1s,qa_2s\end{array};q,q^2\right)
\]
where $x_{1,2}$ and $a_{1,2}$ are the roots of the polynomials
\begin{align*}
1-2xs+s^2=(1- x_1s)(1-x_2s), \quad
1-2xs+r^{-1}s^2=(1-a_1s)(1-a_2s)
\end{align*}
and ${}_3\phi_2$ is the basic hypergeometric function, see \cite[1.10]{b-KLS10}. The second generating function in Proposition \ref{prop:Du-genfn} displays the connection to classical Chebyshev polynomials of the second kind.

\medskip

The above three families of rescaled orthogonal polynomials provide the main ingredients to write down the defining relations \eqref{eq:S=C-intro} in the cases (I) and (II). For any $\field$-algebra $A$ and any noncommutative polynomial $w(x,y)=\sum_{s,t}\lambda_{st} x^s y^t\in A[x,y]$ and any $a_1, a_2,a_3\in A$ we write
\begin{align}\label{eq:zcurve-wxy}
a_3 \curvearrowright w(a_1,a_2) = \sum_{s,t} a_1^{s} a_3 a_2^{t} \lambda_{st}.
\end{align}
Note that the coefficients $\lambda_{st}\in A$ are written at the very right side in \eqref{eq:zcurve-wxy} which will be crucial in case (II).
With these notations we are able to formulate the main result of this paper.
\begin{thm}\label{thm:intro}
  For any symmetrizable Kac-Moody algebra $\gfrak$, generalized Satake diagram $(X,\tau)$ and $\bc\in \cC$ the relations \eqref{eq:S=C-intro} in Theorem \ref{thm:initial-ref} are given as follows:

  \noindent{\bf Cases (I) \& (II):} If $i\in I\setminus X$ with $\tau(i)=i$ and $j\in I$ with $j\neq i$ then
  \begin{align}
    &\sum_{n=0}^{1-a_{ij}}(-1)^n \begin{bmatrix} 1-a_{ij}\\ n\end{bmatrix}_{q_i} B_j \curvearrowright w_{1-a_{ij}-n,n}(B_i, B_i;q_i^2,q_i^{a_{ij}}) \label{eq:casesI+II}\\
      &= \delta_{j\in X}\Bigg(\frac{q_i^{-a_{ij}(a_{ij}+1)}(q_i^2;q_i^2)_{-a_{ij}}}{(q_i-q_i^{-1})^2(q_j-q_j^{-1})}\partial_j^R(\cZ_i)K_j u_{-a_{ij}-1}(B_i;q_i^2,q_i^{2a_{ij}}) \nonumber\\
      & \phantom{= \delta_{j\in X}xx} - \frac{q_i^{a_{ij}(a_{ij}+1)}(q_i^{-2};q_i^{-2})_{-a_{ij}}}{(q_i-q_i^{-1})^2(q_j-q_j^{-1})}
      K_j^{-1 }\partial_j^L(\cZ_i) u_{-a_{ij}-1}(B_i;q_i^{-2},q_i^{-2a_{ij}}) \Bigg).\nonumber
  \end{align}
  \noindent{\bf Case (III):} If $i,j\in I\setminus X$ and $\tau(i)=j\neq i$ then
  \begin{align*}
    S_{ij}(B_i,B_j)=\frac{q_i^{a_{ij}-2} (q_i^{2};q_i^{2})_{1-a_{ij}}}{(q_i-q_i^{-1})^2} B_i^{-a_{ij}} K_j K_i^{-1}\cZ_i
    + \frac{(q_i^{-2};q_i^{-2})_{1-a_{ij}}}{(q_i-q_i^{-1})^2}B_i^{-a_{ij}} K_i K_j^{-1} \cZ_j
  \end{align*}
  In all other cases relation \eqref{eq:S=C-intro} is given by $S_{ij}(B_i,B_j)=0$.
\end{thm}
We will show in Proposition \ref{prop:Serre-bi-uni} that the left hand side of Equation \eqref{eq:casesI+II} can be expressed in terms of rescaled univariate continuous $q$-Hermite polynomials as
\begin{align*}
  \sum_{n=0}^{1-a_{ij}}(-1)^n \begin{bmatrix} 1-a_{ij}\\ n\end{bmatrix}_{q_i} w_{1-a_{ij}-n}(B_i;q_i^2)\,B_j\, w_n(B_i;q_i^{-2}).
\end{align*}
However, the above expression has to be interpreted with care. In case (II), following the convention \eqref{eq:zcurve-wxy}, all powers of $\cZ_i$ for $w_{1-a_{ij}-n}(B_i;q_i^2)$ have to be written to the right side of the factor $B_j$. This subtlety is not relevant in case (I), as in this case $\cZ_i$ and $B_j$ commute.

As mentioned above, the formula in case (III) already appeared in \cite[Theorem 3.6]{a-BalaKolb15}, but we give an independent proof using star products.
\subsection{Methods}
Recall the filtration $\cF$ of $\cB_\bc$ defined by \eqref{eq:filt-def}. The associated graded algebra $\mathrm{gr}(\cB_\bc)$ is isomorphic to the partial parabolic subalgebra
\begin{align*}
  \cA=\field\langle F_i, E_j,K_\beta\,|\,i\in I, j\in X, \beta \in Q^\Theta \rangle \subset \Uq.
\end{align*}
We lift the isomorphism $\gr(\cB_\bc)\cong \cA$ to an explicit $\field$-linear isomorphism $\psi:\cB_\bc \rightarrow \cA$ which we call the \textit{Letzter map}, as a similar map appeared in \cite{a-Letzter19}. We use the Letzter map to push forward the algebra structure of $\cB_\bc$ to a new algebra structure $(\cA,\ast)$ on the vector space $\cA$. The algebra structure $(\cA,\ast)$ is a star product on the $\N$-graded algebra $\cA$ in the sense of \cite[Definition 5.1]{a-KY20}. This means in particular that the algebra $(\cA,\ast)$ has the same generating set as $\cA$ and that its defining relations are those of $\cA$ re-expressed in terms of the star product $\ast$. To obtain the quantum Serre relations \eqref{eq:S=C-intro} for $\cB_\bc$ it hence suffices to rewrite the quantum Serre relations $S_{ij}(F_i,F_j)=0$ for $\cA$ in terms of the star product on $\cA$.

To describe the star product on $\cA$, we use Radford's biproduct decomposition \cite{a-Radford85} for the negative parabolic subalgebra of $\Uq$, to obtain a tensor product decomposition
\begin{align}\label{eq:AHRX}
  \cA\cong \cH\ot \cR_X.
\end{align}
Here $\cH=\cF_0(\cB)=\cA_0$ and $\cR_X$ is a subalgebra of coinvariants which is generated in degree one. The star product on $\cA$ is uniquely determined by the maps 
\begin{align}\label{eq:muLf}
  \mu^L_f(u)=f\ast u - fu \qquad \mbox{for $f\in (\cR_X)_1=\cR_X\cap \cA_1$, $u\in \cR_X$.}
\end{align}
These maps can be expressed in terms of twisted Lusztig-Kashiwara skew-derivations $\partial_{i,X}^{L/R}=T_{w_X}\circ \partial_i^{L/R}\circ T_{w_X}^{-1}$.
In the cases (I) and (II), inductive arguments based on the maps \eqref{eq:muLf} show that
\begin{align}\label{eq:FimFjFin-intro}
  F_i^mF_jF_i^n=& F_j\curvearrowright w_{m,n}(F_i\stackrel{\ast}{,}F_i;q_i^2,q_i^{a_{ij}})\\
  &\qquad + \delta_{j\in X}\Big(\partial^R_j(\cZ_i)K_j \rho_{m,n}(F_i)^\ast + K_j^{-1}\partial_j^L(\cZ_i)\sigma_{m,n}(F_i)^\ast\Big)\nonumber
\end{align}
for some polynomials $\rho_{m,n}(x), \sigma_{m,n}(x)\in \cM_X^+[x]$, see Proposition \ref{prop:FimFjFin} and Lemma \ref{lem:Fiast-j}. The symbol $\ast$ appearing on the right hand side of the above equation indicates that all polynomials are evaluated in $(\cA,\ast)$, see Section \ref{sec:o-p} for the precise notation.
To obtain formula \eqref{eq:casesI+II} we hence need to determine the quantum Serre combination
\begin{align}\label{eq:qSerre-sigma}
  \sum_{n=0}^{1-a_{ij}}(-1)^n \begin{bmatrix} 1-a_{ij}\\ n\end{bmatrix}_{q_i} \rho_{1-a_{ij}-n}(F_i)^\ast
\end{align}
and similarly for $\sigma_{m,n}$. The inductive argument from Lemma \ref{lem:Fiast-j} provides recursive formulas for $\rho_{m,n}(x)$ and $\sigma_{m,n}(x)$. Using a generalization of the bar-involution for $\cB_\bc$, which was implicit in \cite{a-AV20p} and has been made explicit in \cite{a-Kolb21p}, we show that it suffices to consider the quantum Serre combination \eqref{eq:qSerre-sigma}. In Section \ref{sec:PN} we determine \eqref{eq:qSerre-sigma} using a generating function approach which naturally leads to the deformed Chebyshev polynomials of the second kind. 
\subsection{Organization of the paper}
In Section \ref{sec:QSP-PPS} we fix notation and present background material on quantum groups and quantum symmetric pairs. In particular, in Sections \ref{sec:parabolic} and \ref{sec:partial-parabolic} we recall how Radford's biproduct leads to the tensor product decomposition \eqref{eq:AHRX}. Moreover, we define the Letzter map in Section \ref{sec:Letzter} and recall the generalized bar-involution in Section \ref{sec:anti-iso}.

Section \ref{sec:star-para} develops the star product on the partial parabolic subalgebra $\cA$. After recalling the definition of a star product in Section \ref{sec:star-def}, we determine the maps $\mu_f^L$ for $(\cA,\ast)$ in Section \ref{sec:Bc-star}. In Section \ref{sec:twisted-skew} we then determine the values of the twisted skew-derivations $\partial^{L/R}_{i,X}$ which will be needed later to derive Equation \eqref{eq:FimFjFin-intro}.

In Section \ref{sec:orth-polyn} we discuss the three families of orthogonal polynomials given in A), B) and C) above. We first recall the univariate and bivariate continuous $q$-Hermite polynomials and list essential properties of their rescaled versions in Section \ref{sec:biv-q-Herm}. In Section \ref{sec:wmn-Serre-combi} we show how the quantum Serre combination of $w_{m,n}(x,y;q_i^2,q_i^{2a_{ij}})$ can be expressed in terms of $w_m(x;q_i^2)$ and $w_m(x;q_i^{-2})$. In Section \ref{sec:Chebyshev} we introduce the deformed Chebyshev polynomials of the second kind and determine two generating functions for them.

The final Section \ref{sec:GenRel} contains the proof of Theorem \ref{thm:intro}. In Section \ref{sec:caseI} we derive Equation \eqref{eq:casesI+II} in case (I). In Section \ref{sec:jinX} we derive Equation \eqref{eq:FimFjFin-intro} in the subtle case (II). The quantum Serre combination \eqref{eq:qSerre-sigma} is determined in Section \ref{sec:PN}. This allows us to prove formula \eqref{eq:casesI+II} also in case (II). Finally, Section \ref{sec:caseIII} contains the proof of case (III) of Theorem \ref{thm:intro}.
\section{Quantum symmetric pairs and partial parabolic subalgebras}\label{sec:QSP-PPS}
In this introductory section we expand on the notation given in Section \ref{sec:problem} and recall some standard constructions for quantized enveloping algebras, in particular the decomposition of a quantized standard parabolic as a Radford biproduct \cite{a-Radford85}. We then recall that quantum symmetric pair coideal subalgebras $\cB_\bc$ have a filtration such that the associated graded algebra is isomorphic to a large subalgebra $\cA$ of a quantized standard parabolic. This isomorphism can be lifted to a filtered linear isomorphism $\psi:\cB_\bc\rightarrow \cA$ which we call the Letzter map. This map provides the interpretation of $\cB_\bc$ as a star product deformation of $\cA$ in Section \ref{sec:star-para}.
\subsection{The setting}\label{sec:setting}
Let $\gfrak$ be a symmetrizable Kac-Moody algebra with generalized Cartan matrix $(a_{ij})_{i,j\in I}$ where $I$ is a finite set. Let $\{d_i\,|\,i\in I\}$ be a set of relatively prime positive integers such that the matrix $(d_ia_{ij})$ is symmetric. Let $\Pi=\{\alpha_i\,|\,i\in I\}$ be the set of simple roots for $\gfrak$ and let $Q=\Z\Pi$ be the root lattice. Consider the symmetric bilinear form $(\cdot,\cdot):Q\times Q \rightarrow \Z$ defined by $(\alpha_i,\alpha_j)=d_ia_{ij}$ for all $i,j\in I$. Let $W$ be the corresponding Weyl group with simple reflections $s_i$ for $i\in I$. Let $\gfrak'=[\gfrak,\gfrak]$ be the derived subalgebra of $\gfrak$.

Throughout the paper we fix a field $k$ of characteristic 0 and set $\field=k(q)$. The quantized enveloping algebra $U=U_q(\gfrak')$ is the $\field$-algebra with generators $E_i, F_i, K_i^{\pm 1}$ for $i \in I$ and defining relations given in \cite[3.1]{b-Lusztig94}. In particular, the generators $E_i, F_i$ satisfy the quantum Serre relations
\begin{align*}
     S_{ij}(E_i,E_j)=0=S_{ij}(F_i,F_j) 
\end{align*}
for all $i,j\in I$, where
\begin{align}\label{eq:qSerre}
    S_{ij}(x,y)=\sum_{\ell=0}(-1)^\ell \stirling{1-a_{ij}}{\ell}_{q_i}x^{1-a_{ij}-\ell}y x^\ell \in \field[x,y]
\end{align}
with $q_i=q^{d_i}$ denotes the quantum Serre polynomial, see \cite[Corollary 33.1.5]{b-Lusztig94}.
The algebra $\Uq$ is a Hopf algebra with coproduct $\kow$ given by
 \begin{align}\label{eq:kow-def}
    \kow(E_i){=}E_i\ot 1+ K_i\ot E_i, \quad \kow(F_i){=}F_i\ot K_i^{-1}+1 \ot F_i, \quad \kow(K_i){=}K_i\ot K_i
 \end{align}
for all $i\in I$. For $\beta=\sum_{i\in I}n_i \alpha_i\in Q$ write $K_\beta=\prod_{i\in I}K_i^{n_i}$.

We write $\N=\{0,1,2,\dots\}$ and let $Q^+=\N\Pi$ be the positive cone in the root lattice. For any $\mu\in Q^+$ we let $U^+_\mu=\mathrm{span}_\field\{E_{i_1} \dots E_{i_\ell}\,|\, \sum_{j=1}^\ell \alpha_{i_j}=\mu\}$ be the corresponding root space, where $\mathrm{span}_\field$ denotes the $\field$-linear span. Moreover, define $U^-_{-\mu}=\omega(U^+_\mu)$ where $\omega:U\rightarrow U$ is the algebra automorphism given in \cite[3.1.3]{b-Lusztig94}. We will use the Lusztig--Kashiwara skew-derivations $\partial^R_i, \partial^L_i:U^\pm\rightarrow U^\pm$ for $i\in I$, which are uniquely determined by
\begin{align}
    E_i y - y E_i &=\frac{K_i \partial^L_i(y) - \partial^R_i(y)K_i^{-1}}{q_i-q_i^{-1}} \label{eq:EyyE}, \\
    F_i x - x F_i &= \frac{K_i^{-1} \partial^L_i(x) - \partial^R_i(x)K_i}{q_i-q_i^{-1}}\label{eq:FxxF}
\end{align}
for $x\in U^+$ and $y\in U^-$, see \cite[Proposition 3.1.6]{b-Lusztig94}. The skew-derivations $\partial_i^L$ and $\partial_i^R$ satisfy the relations
\begin{align}
   \partial_i^L(fg)&=\partial_i^L(f)g +q^{(\alpha_i,\mu)}f \partial_i^L(g),\label{eq:partialL}\\
   \partial_i^R(fg)&=q^{(\alpha_i,\nu)}\partial_i^R(f)g +f \partial_i^R(g)\label{eq:partialR}
\end{align}
for all $f\in U^-_{-\mu}$, $g\in U^-_{-\nu}$ and for all $f\in  U^+_{\mu}$, $g\in U^+_{\nu}$. 

For any $i\in I$ let $T_i:U\rightarrow U$ be the algebra automorphism denoted by $T''_{i,1}$ in \cite[37.1]{b-Lusztig94}. The skew-derivations  $\partial_i^L$ and $\partial_i^R$ detect elements in the intersection $U^\pm \cap T_i(U^{\pm})$. More precisely, for $x\in U^+$ and $y\in U^-$ we have the following equivalences:
\begin{align}
    x&\in T_i^{-1}(U^+) & &\Longleftrightarrow & \partial_i^L(x)&=0,\label{eq:T_i-equivalence+L}\\
    x&\in T_i(U^+) & &\Longleftrightarrow & \partial_i^R(x)&=0,\label{eq:T_i-equivalence+R}\\
    y&\in T_i^{-1}(U^-) & &\Longleftrightarrow & \partial_i^L(y)&=0,\label{eq:T_i-equivalence}\\
    y&\in T_i(U^-) & &\Longleftrightarrow & \partial_i^R(y)&=0,\label{eq:T_i-equivalence-R}
\end{align}
see \cite[38.1.6, 37.2.4]{b-Lusztig94}. 

Given a subset $X \subseteq I$, denote 
\[
Q_X=\sum_{j\in X} \Z \alpha_j \quad \mbox{and} \quad Q_X^+=\sum_{j\in X}\N \alpha_j.
\]
We write $\cM_X$ to denote the subalgebra of $U$ generated by $\{E_j, F_j, K_j^{\pm 1}\,|\,j\in X\}$, and $\cM^+_X$ and 
$\cM^-_X$ for the subalgebras generated by $\{E_j\,|\,j\in X\}$ and $\{F_j\,|\,j\in X\}$, respectively.
  We write $G_X^+$ and $G_X^-$ to denote the subalgebras of $U$ generated by $\{E_jK_j^{-1}\,|\,j\in X\}$ and $\{F_jK_j\,|\,j\in X\}$, respectively. Moreover, we write $G^\pm$ for $G^\pm_I$. 
\subsection{Parabolic subalgebras}\label{sec:parabolic}
Let $X\subseteq I$ be a subset of finite type. The parabolic subalgebra
\begin{align*}
  \cP_X=\field\langle F_i, E_j, K_i^{\pm 1}\,|\, i\in I, j\in X\rangle
\end{align*}
is a Hopf subalgebra of $\Uq$ and the corresponding Levi factor
\begin{align*}
  \cL_X=\field\langle F_j, E_j, K_i^{\pm 1}\,|\, i\in I, j\in X\rangle
\end{align*}
is a Hopf subalgebra of $\cP_X$. The parabolic subalgebra and the Levi factor have triangular decompositions
\begin{align}\label{eq:PL-triang}
  \cP_X\cong \cM_X^+ \ot U^0 \ot U^-, \qquad \cL_X \cong \cM_X^+\ot U^0 \ot \cM_X^-, 
\end{align}
respectively. There is a surjective Hopf algebra homomorphism $\pi_X:\cP_X\rightarrow \cL_X$ defined by
\begin{align*}
   \pi_X|_{\cL_X}=\id_{\cL_X} \quad \mbox{and} \quad \pi_X(F_i)= 0 \qquad \mbox{for all $i\in I\setminus X$.} 
\end{align*}
The structure of Hopf algebras with a projection onto a Hopf subalgebra was investigated in detail by D.~Radford in \cite{a-Radford85}. In the following we recall some of his results in our setting. All the material of this section is known to the experts, but we include some proofs for the convenience of the reader.

Consider the left coaction of $\cL_X$ on $\cP_X$ given by
\begin{align*}
  \kow_{\cL_X}=(\pi_X \ot \id)\circ \kow: \cP_X \rightarrow \cL_X\otimes \cP_X
\end{align*}  
and define a subalgebra $\cR_X\subset \cP_X$ by
\begin{align*}
  \cR_X={}^{\cL_X}\cP_X=\{a\in \cP_X\,|\, \kow_{\cL_X}(a)=1 \ot a\}. 
\end{align*}
The algebra $\cR_X$ is is a right $\cL_X$-module algebra under the right adjoint action given by $\ad_r(h)(a)= S(h_{(1)})a h_{(2)}$ for all $a\in\cR_X, h\in \cL_X$. Hence we can form the smash product $\cL_X\ot \cR_X$, which is an associative algebra with the multiplication
\begin{align*}
  (h\ot a)(h'\ot a') = h h'_{(1)} \ot \ad_r(h'_{(2)})(a)a'.
\end{align*}
The following Theorem is obtained by translating \cite[Theorem 3(d)]{a-Radford85} from left to right.
\begin{thm}\label{thm:Radford}
  The multiplication map $\cL_X \ot \cR_X \rightarrow \cP_X$ is an isomorphism of algebras.
\end{thm}  
\begin{rema}
  The statement of \cite[Theorem 3]{a-Radford85} goes beyond the above theorem. The algebra $\cR_X$ is a Hopf algebra in the category of right Yetter-Drinfeld modules over $\cL_X$. One can hence form Radford's biproduct $\cL_X\times \cR_X$ (also known as the bosonization \cite{a-Majid94}), which is a Hopf algebra that coincides with the smash product $\cL_X\ot \cR_X$ as an algebra. By \cite[Theorem 3(d)]{a-Radford85} the multiplication map  $\cL_X \times \cR_X \rightarrow \cP_X$ is an isomorphism of Hopf algebras. We will only need the algebra structure on $\cL_X\ot \cR_X$.
\end{rema}
The formula \eqref{eq:kow-def} for the coproduct of $F_i$ implies that $F_i\in \cR_X$ for all $i\in I\setminus X$. Moreover, as $\cR_X$ is invariant under the right adjoint action of $\cL_X$, we obtain that $\ad_r(\cL_X)(F_i) \subseteq \cR_X$ for all $i\in I\setminus X$.
\begin{cor}\label{cor:VXRX}
  The algebra $\cR_X$ is generated by the subspaces $\ad_r(\cL_X)(F_i)$ for all $i\in I\setminus X$. 
\end{cor}  
\begin{proof}
  Let $V_X\subset \cP_X$ be the subalgebra generated by the subspaces $\ad_r(\cL_X)(F_i)$ for all $i\in I\setminus X$. As observed above we have $V_X\subseteq \cR_X$. The subalgebra of $\cP_X$ generated by $V_X$ and $\cL_X$ coincides with $\cP_X$. Hence the formula $a h = h_{(1)} \ad_r(h_{(2)})(a)$ for $a\in V_X$, $h\in \cL_X$ implies that the multiplication map $\cL_X\ot V_X\rightarrow \cP_X$ is surjective. Now the inclusion $V_X\subseteq \cR_X$ and Theorem \ref{thm:Radford} imply that $V_X=\cR_X$.
\end{proof}  
For all $j\in X$, $i\in I\setminus X$ one has $\ad_r(E_j)(F_i)=-K_j^{-1}[E_j,F_i]=0$. Hence by the triangular decomposition \eqref{eq:PL-triang} for $\cL_X$ we get $\ad_r(\cL_X)(F_i)=\ad_r(\cM^-_X)(F_i)\subset U^-$. By the above theorem and corollary and the triangular decompositions \eqref{eq:PL-triang} we hence get that the multiplication map
\begin{align}\label{eq:MRU-}
  \cM_X^- \ot \cR_X \rightarrow U^-
\end{align}
is an isomorphism of algebras.

We can also describe the subalgebra $\cR_X\subset U^-$ in terms of the Lusztig automorphisms.
\begin{lem} \label{lem:RTwX}
  We have
  \begin{align*}
    \cR_X&=\{a\in U^-\,|\, T_j^{-1}(a)\in U^- \mbox{ for all $j\in X$}\}=\bigcap_{j\in X} \big(U^-\cap T_j(U^-)\big).
  \end{align*}  
\end{lem}
\begin{proof}
  Let $j\in X$. We first show that $T_j^{-1}(\cR_X)\subset U^-$. By Corollary \ref{cor:VXRX} it suffices to show that $T_j^{-1}(\ad_r(\cM_X^-)(F_i))\subset U^-$ for any $i\in I\setminus X$. We have $T_j^{-1}(F_i)\in U^-$. Now we proceed by induction. Assume that $u\in \cR_X$ satisfies $T_j^{-1}(u)\in U^-$ and let $\ell \in X$. Then we get 
  \begin{align*}
    T_j^{-1}(\ad_r(F_\ell)(u))&=T_j^{-1}(-F_\ell K_\ell u K_{\ell}^{-1} + u F_\ell)
  \end{align*}
  which by induction hypothesis lies in $U^-$ if $\ell\neq j$. In the case $\ell=j$ the above formula becomes
  \begin{align*}
     T_j^{-1}(\ad_r(F_\ell)(u))=[T_\ell^{-1}(u), E_\ell] K_\ell \stackrel{\eqref{eq:T_i-equivalence}}{=} \frac{1}{q_\ell-q_\ell^{-1}}\partial_\ell^R(T_\ell^{-1}(u))\in U^-.
  \end{align*}
 This implies that $T_j^{-1}(\ad_r(\cM_X^-)(F_i))\subseteq U^-$ as required.

 Conversely, assume that $u\in U^-$ satisfies $T_j^{-1}(u)\in U^-$ for all $j\in X$. We want to show that $u\in \cR_X$. For any $j\in X$ the decomposition \eqref{eq:MRU-} implies that the multiplication map
 \begin{align}\label{eq:MX-decomp}
   \field[F_j]\ot (\cM_X^-\cap T_j(M_X^-))\ot \cR_X\rightarrow U^-
 \end{align}
 is a linear isomorphism. Write $u=\sum_m h_m\ot v_m$ with linearly independent $h_m\in \cM_X^-\cong \field[F_j]\ot (\cM_X^-\cap T_j(\cM_X^-))$ and $v_m\in \cR_X$. Then the relations $T_j^{-1}(u)\in U^-$ and $T_j^{-1}(v_m)\in U^-$ imply that $h_m\in (\cM_X^-\cap T_j(\cM_X^-))$. By \cite[Lemma 1.2.15]{b-Lusztig94} we have
 \begin{align*}
    \bigcap_{j\in X} (\cM_X^-\cap T_j(\cM_X^-))\stackrel{\eqref{eq:T_i-equivalence-R}}{=} \{y\in \cM_X^-\,|\,\partial^R_j(y)=0 \mbox{ for all $j\in X$}\} = \field.
 \end{align*}
Hence all $h_m$ are scalars which implies that $u\in \cR_X$. 
\end{proof}
\begin{rema}
  In Equation \eqref{eq:MX-decomp} we used the decomposition $\cM_X^-\cong \field[F_j]\ot (\cM_X^-\cap T_j(\cM_X^-))$ which holds by the PBW Theorem for the finite type quantum enveloping algebra $\cM_X$. By the above lemma and the decomposition \eqref{eq:MRU-} in the case where $|X|=1$, we now have the same decomposition
  \begin{align}\label{eq:FiU-TiU-}
     U^-\cong \field[F_i] \ot (U^-\cap T_i(U^-))
  \end{align}
  for any $i\in I$ also in infinite type.
\end{rema}  
To rewrite the statement of Lemma \ref{lem:RTwX} we note the following fact.
\begin{lem}
  Let $w\in W$ be an element with reduced expression $w=s_{i_1}s_{i_2} \dots s_{i_m}$. Then the following relations hold
  \begin{align}
    U^- \cap T_w(U^-)&= U^- \cap T_{i_1}(U^-)\cap T_{i_1} T_{i_2}(U^-) \cap \dots \cap T_w(U^-),\label{eq:U-capTwU-1}\\
    U^- \cap T_w^{-1}(U^-)&= U^- \cap T^{-1}_{i_m}(U^-)\cap T^{-1}_{i_m} T^{-1}_{i_{m-1}}(U^-) \cap \dots \cap T^{-1}_w(U^-).\label{eq:U-capTwU-2}
  \end{align}  
\end{lem}
\begin{proof}
  It suffices to prove  \eqref{eq:U-capTwU-2} as  \eqref{eq:U-capTwU-1} then follows by application of $T_w$. We prove the equality \eqref{eq:U-capTwU-2} by induction on the length $l(w)=m$ of $w$. For $m=1$ there is nothing to show. Now assume that $w=s_i w'$ where $l(w)=l(w')+1$. Then $T_w=T_i T_w'$ and Equation \eqref{eq:FiU-TiU-} gives us
  \begin{align*}
    T^{-1}_w(U^-)&=T^{-1}_{w'} T^{-1}_i \big(\field[F_i] \ot (U^-\cap T_i(U^-))\big)\\
    &= T_{w'}^{-1}\big(\field[E_iK_i] \ot (T_i^{-1}(U^-)\cap U^-) \big)\\
    &= \field[T_{w'}^{-1}(E_iK_i)] \ot \big(T_w^{-1}(U^-) \cap T_{w'}^{-1}(U^-)\big).
  \end{align*}
  By \cite[Proposition 40.2.1]{b-Lusztig94} we have $T_{w'}^{-1}(E_i)\in U^+$ and hence the triangular decomposition $U^+\ot U^0 \ot U^-\cong U$ implies that
  \begin{align*}
    U^- \cap T^{-1}_w(U^-) = U^- \cap T_w^{-1}(U^-) \cap T_{w'}^{-1}(U^-).
  \end{align*}
  Now the equality \eqref{eq:U-capTwU-2} follows by induction hypothesis. 
\end{proof}  
Let $w_X$ denote the longest element in the finite parabolic subgroup $W_X$ of the Weyl group $W$. There is a diagram automorphism $\tau_X:X\rightarrow X$ such that $w_X(\alpha_j)=-\alpha_{\tau_X(j)}$ for all $j\in X$.
\begin{cor}\label{cor:RTwX}
  We have $\cR_X=U^-\cap T_{w_X}(U^-)$.
\end{cor}  
\begin{proof}
  For any $j\in X$ we can write $w_X=s_j w'$ for some $w'\in W_X$ with $l(w')=l(w_X)-1$. Hence Equation \eqref{eq:U-capTwU-1} and Lemma \ref{lem:RTwX} imply that
  \begin{align*}
     U^-\cap T_{w_X}(U^-) \subseteq \bigcap_{j\in X} (U^- \cap T_j(U^-))=\cR_X.
  \end{align*}
  The converse inclusion is verified similarly to the first part of the proof of Lemma \ref{lem:RTwX}. Again it suffices to show that $\ad_r(\cM_X^-)(F_i) \subseteq U^-\cap T_{w_X}(U^-)$ for all $i\in I\setminus X$. By \cite[Proposition 40.2.1]{b-Lusztig94} we have $T_{w_X}^{-1}(F_i)\in U^-$ and  hence $F_i\in U^-\cap T_{w_X}(U^-)$. If $u\in U^-\cap T_{w_X}(U^-)$ and $\ell\in X$ then
  \begin{align*}
    T_{w_X}^{-1}(\ad_r(F_\ell)(u))=[T_{w_X}^{-1}(u), E_{\tau_X(\ell)}]K_{\tau_X(\ell)}  \stackrel{\eqref{eq:T_i-equivalence}}{=} \frac{1}{q_\ell-q_\ell^{-1}} \partial_{\tau_X(\ell)}^R(T_{w_X}^{-1}(u))\in U^-.
  \end{align*}
  This implies inductively that $T_{w_X}^{-1}(\ad_r(\cM_X^-)(F_i))\subset U^-$ and completes the proof of the corollary.
\end{proof}  

\subsection{Partial parabolic subalgebras}\label{sec:partial-parabolic}
Consider a subset of finite type $X\subseteq I$ and a map $\tau:I\rightarrow I$ which is an involutive diagram automorphism such that $\tau(X)=X$ and $w_X(\alpha_j)=-\alpha_{\tau(j)}$ for all $j\in X$. 
Following \cite{a-RV20}, we call a pair $(X,\tau)$ with these properties a compatible decoration. The map $\Theta=-w_X\circ \tau:Q\rightarrow Q$ is an involutive automorphism of the root lattice. Define $Q^\Theta=\{\beta\in Q\,|\,\Theta(\beta)=\beta\}$ and $U^0_\Theta=\field\langle K_\beta\,|\,\beta\in Q^\Theta\rangle$. As $Q^\Theta$ is generated by the elements of the set $\{\alpha_i-\alpha_{\tau(i)},\alpha_j\,|\,i\in I\setminus X, j\in X\}$, we see that
\begin{align*}
   U^0_\Theta=\field \langle K_j^{\pm 1}, K_i K_{\tau(i)}^{-1}\,|\,j\in X, i\in I \rangle.
\end{align*}
Consider the subalgebra $\cA=\cA(X,\tau)\subset U$ defined by
\begin{align*}
  \cA=\field\langle F_i, E_j, K_\beta\,|\, i\in I, j\in X, \Theta(\beta)=\beta\rangle.
\end{align*}
We call $\cA$ the partial parabolic subalgebra corresponding to the pair $(X,\tau)$. By definition, $\cA$ is contained in the standard parabolic subalgebra $\cP_X$, with the only difference that $\cA\cap U^0=U^0_\Theta$ is not the whole torus. In particular, $\cA$ has the triangular decomposition
\begin{align}\label{eq:A-triang}
  \cA\cong \cM_X^+ \ot U^0_\Theta \ot U^-.
\end{align}  
Additionally, consider the subalgebra $\cH=\cH(X,\tau)\subset U$ defined by
\begin{align*}
  \cH=\field \langle E_j, F_j, K_\beta\,|\,j\in X, \Theta(\beta)=\beta\rangle=\cM_X U^0_\Theta.
\end{align*}  
We call $\cH$ the partial Levi factor corresponding to the generalized Satake diagram $(X,\tau)$. The partial Levi factor $\cH$ has a triangular decomposition
\begin{align}\label{eq:cH-triang}
  \cH \cong \cM^+_X \ot U_\Theta^0 \ot \cM^-_X
\end{align}  
The partial parabolic $\cA$ is not a subbialgebra of $U$. However, the partial  Levi factor $\cH$ is a Hopf subalgebra of $U$. Comparison of the triangular decompositions \eqref{eq:A-triang} and \eqref{eq:cH-triang} with the decomposition \eqref{eq:MRU-} implies that the multiplication map
\begin{align*}
  \cH \ot \cR_X \rightarrow \cA
\end{align*}
is an isomorphism. 
\subsection{Quantum symmetric pairs}\label{sec:QSP}
Let $(X,\tau)$ be an generalized Satake diagram. Following \cite{a-RV20} this means that $(X,\tau)$ is a compatible decoration, as described at the beginning of Section \ref{sec:partial-parabolic}, and additionally, that for $i\in I\setminus X$ and $j\in X$ the implication
\begin{align}\label{eq:generalizedSatake}
  \tau(i)=i \mbox{ and } a_{ji}=-1 \quad \Longrightarrow \quad \Theta(\alpha_i)\neq -\alpha_i-\alpha_j
\end{align}  
holds. The condition \eqref{eq:generalizedSatake} is weaker than condition (3) in \cite[Definition 2.3]{a-Kolb14}. As explained in \cite[Section 4.1]{a-RV20}, the results of \cite{a-Kolb14} remain valid in this more general setting. Following \cite{a-Kolb14} we recall the definition of the quantum symmetric pair coideal subalgebra corresponding to the generalized Satake diagram $(X,\tau)$. Recall the set of parameters $\cC$ defined by \eqref{eq:C-def-intro}.
For any $\bc=(c_i)_{i\in I\setminus X}\in \cC$ let $\cB_\bc$ denote the subalgebra of $U$ generated by the partial Levi factor $\cH$ and the elements
\begin{align}\label{eq:Bi-def}
  B_i=F_i - c_i T_{w_X}(E_{\tau(i)}) K_i^{-1}.
\end{align}
\begin{rema}
   In \cite{a-Kolb14}, quantum symmetric pairs depend of a second family of parameters $\bs=(s_i)_{i\in I\setminus X}$ in a certain subset $\cS\subset \field^{I\setminus X}$. The corresponding coideal subalgebras are then denoted by $\cB_{\bc,\bs}$. By \cite[Theorem 7.1]{a-Kolb14}, for any $\bs\in \cS$ the algebra $\cB_{\bc,\bs}$ is isomorphic to $\cB_{\bc}$. As we only aim to establish defining relations for $\cB_{\bc,\bs}$, it suffices to consider the case where $\bs=(0,0, \dots,0)$. 
\end{rema}  
Recall that the algebra $\cB_\bc$ has a natural $\N$-filtration $\cF_\ast$ defined via the degree function given by \eqref{eq:filt-def}.
Hence $\cF_n(\cB_\bc)$ is the $\field$-linear span of all monomials in the generators of $\cB_\bc$ where each monomial contains at most $n$ factors $B_i$ with $i\in I\setminus X$.
Let $\gr(\cB_\bc)$ be the associated graded algebra. It follows from \cite[Theorem 7.1]{a-Kolb14} that there exists a surjective algebra homomorphism
\begin{align}\label{eq:varphi}
  \varphi:\cA \rightarrow \gr(\cB_\bc)
\end{align}
given by $\varphi(F_i)=B_i$ for all $i\in I\setminus X$ and $\varphi(h)=h$ for all $h\in \cH$. Moreover, the triangular decomposition \eqref{eq:A-triang} for $\cA$ together with \cite[Proposition 6.2]{a-Kolb14} imply that $\varphi$ is an isomorphism.

The partial parabolic subalgebra $\cA$ is $\N$-graded via the degree function given by
  \begin{align*}
     \deg(F_i)&=1 \qquad \mbox{for all $i\in I\setminus X$,}\\
     \deg(h)&=0 \qquad \mbox{for all $h\in \cH$}.
  \end{align*}
With this grading the map $\varphi$ in \eqref{eq:varphi} is an isomorphism of graded algebras. For any $n\in \N$ we let $\cA_n$ denote the $n$-th graded component of $\cA$ and we write $\cA_{<n}=\bigoplus_{m=0}^{n-1}\cA_m$.  
\subsection{The Letzter map}\label{sec:Letzter}
Let $\Upoly$ denote the subalgebra of $U$ generated by $\cA$ and the elements
$\Etil_i=E_iK_i^{-1}$, $K_i^{-1}$ for all $i\in I\setminus X$. Observe that $\cB_\bc\subset \Upoly$ as $T_{w_X}(E_{\tau(i)})K_i^{-1}\in \cA \Etil_{\tau(i)}\cA$. Let $I_\tau$ be a set of representatives of $\tau$-orbits in $I\setminus X$. In analogy to the definition of $V_X$ in the proof of Corollary \ref{cor:VXRX}, let $V_X^+\subset \cP_X$ be the subalgebra generated by the subspaces $\ad_r(\cL_X)(\Etil_i)=\ad_r(\cM_X^+)(\Etil_i)$ for $i\in I\setminus X$.
The triangular decomposition
\begin{align*}
  U\cong U^- \ot U^0 \ot G^+ \cong \cA \ot \field[K_i^{\pm 1}\,|\,i\in I_\tau]\ot V_X^+
\end{align*}
implies that there is a direct sum decomposition of vector spaces
\begin{align}\label{eq:Upoly-decomp}
  \Upoly \cong \cA \oplus \Upoly \mathrm{span}_\field\{\Etil_i, K_i^{-1}\,|\,i\in I\setminus X\} \cH.
\end{align}
Let
\begin{align}\label{eq:psi-Upoly}
  \psi:\Upoly \rightarrow \cA
\end{align}
denote the $\field$-linear projection with respect to the direct sum decomposition \eqref{eq:Upoly-decomp}.
\begin{rema}
  For $\gfrak$ of finite type, the map $\psi$ is the restriction to $\Upoly$ of a projection map $\cP:\uqg \rightarrow \cA$ given in \cite[(4.9)]{a-Letzter19}.
\end{rema}
As $\cB_\bc$ is a subalgebra of $\Upoly$  we can restrict the map $\psi$ to $\cB_\bc$. We call the map $\psi:\cB_\bc\rightarrow \cA$ the \textit{Letzter map}. The algebra $\cB_\bc$ is filtered by the filtration $\cF_\ast$ given in Section \ref{sec:QSP}. The algebra $\cA$ is graded and hence also filtered. By construction the Letzter map $\psi:\cB_\bc\rightarrow \cA$ is a linear map of filtered vector spaces. Let
\begin{align*}
  \gr(\psi):\gr(\cB_\bc)\rightarrow \gr(\cA)\cong \cA
\end{align*}
be the associated graded map and recall the isomorphism $\varphi$ from \eqref{eq:varphi}. The composition $\gr(\psi)\circ \varphi:\cA \rightarrow \cA$ is $\cH$-linear and satisfies
\begin{align*}
  \gr(\psi)\circ \varphi(F_{i_1}\dots F_{i_\ell})= F_{i_1}\dots F_{i_\ell}
\end{align*}  
for all $i_1,\dots, i_\ell\in I$. Hence $\gr(\psi)\circ \varphi=\id_\cA$. This implies that $\gr(\psi)=\varphi^{-1}$ is a linear isomorphism. This proves the following lemma which is also contained in \cite[Corollary 4.4]{a-Letzter19} for $\gfrak$ of finite type.
  \begin{lem}\label{lem:letzter}
    The Letzter map $\psi:\cB_\bc\rightarrow \cA$ is a linear isomorphism. 
  \end{lem}
  We collect some additional properties of the Letzter map.
  \begin{lem}\label{lem:psi-props}
    The Letzter map $\psi:\cB_\bc\rightarrow \cA$ has the following properties:
    \begin{enumerate}
      \item $\psi(ab) = \psi(a\psi(b))$ for all $a,b\in \cB_\bc$.
      \item $\psi(h_1bh_2) = h_1\psi(b) h_2$ for all $b\in \cB_\bc$, $h_1,h_2\in \cH$.
      \item $\psi(T_i(b))=T_i(\psi(b))$ for all $i\in X$.
      \item $\psi(ab)-\psi(a)\psi(b)\in \cA_{<m+n}$ for all $a\in \cF_m(\cB_\bc)$, $b\in \cF_n(\cB_\bc)$.  
    \end{enumerate}  
  \end{lem}
  \begin{proof}  
    The kernel of the projection map \eqref{eq:psi-Upoly} is a left ideal in $\Upoly$. Hence the kernel of the Letzter map is a left ideal in $\cB_\bc$. This implies statement (1). Similarly, statement (2) follows from the fact that the decomposition \eqref{eq:Upoly-decomp} is a sum of $\cH$-bimodules. And statement (3) holds because both summands in the decomposition \eqref{eq:Upoly-decomp} are invariant under $T_i$ for $i\in X$.

Recall from the discussion above Lemma \ref{lem:letzter} that the Letzter map $\psi:\cB_\bc\rightarrow \cA$ is a linear isomorphism of filtered vector spaces, and that the associated graded map $\gr(\psi):\gr(\cB_\bc)\rightarrow \cA$ is an algebra isomorphism. This implies property (4).   
  \end{proof}
  \subsection{An antilinear isomorphism of quantum symmetric pair coideals}\label{sec:anti-iso}
Let $\overline{\phantom{m}}:\Uq\rightarrow \Uq$, $u\mapsto \overline{u}$ be the Lusztig bar involution defined in \cite[3.1.12]{b-Lusztig94}. 
Denote by $\rho_X$ and $\rho_X^\vee$ the half sum of positive roots and coroots, respectively, for the root system corresponding to $X \subseteq I$. 
For any $\bc=(c_i)_{i\in I\setminus X}\in \field^{I\setminus X}$ define $\bc'=(c_i')\in \field^{I\setminus X}$ by
\begin{align}\label{eq:c'-def}
     c_i'=(-1)^{2\alpha_i(\rho_X^\vee)}q^{(\alpha_i,\Theta(\alpha_i)-2\rho_X)} \overline{c_{\tau(i)}}
\end{align}
for all $i\in I\setminus X$. Note that $\bc\in \cC$ if and only if $\bc'\in \cC$. In the  following we will compare the quantum symmetric pair coideal subalgebras $\cB_\bc$ and $\cB_{\bc'}$. To indicate the parameters, we write $B_i'$ for the generators \eqref{eq:Bi-def} of $\cB_{\bc'}$ for $i\in I\setminus X$.
  
In \cite[Theorem 7.4]{a-AV20p} Appel and Vlaar considered an extended version of the quasi K-matrix constructed in \cite{a-BalaKolb19}. 
Using this extended version of the quasi K-matrix one obtains the following theorem. A proof the theorem in the notations of the present paper will appear in \cite{a-Kolb21p}.
\begin{thm}\label{thm:bar}
  For any $\bc\in \cC$ there exists a $k$-algebra isomorphism $\Phi:\cB_\bc\rightarrow \cB_{\bc'}$ such that
  \begin{align*}
     \Phi|_\cH=\overline{\phantom{m}}|_\cH \qquad \mbox{and} \qquad \Phi(B_i)=B_i' \quad \mbox{for all $i\in I\setminus X$.}
  \end{align*}
In particular, we have $\Phi(q)=q^{-1}$ and $\Phi(K_\beta)=K_{-\beta}$ for all $\beta\in Q^\Theta$.
\end{thm}
\section{The star product on the partial parabolic subalgebra}\label{sec:star-para}
In this section we recall the general notion of a star product on an $\N$-graded algebra $A$. We then recall the fact that $\cB_\bc$ can be viewed as a star product deformation of the partial parabolic subalgebra $\cA$. To describe this star product, we need to know the action of certain twisted skew-derivations $\partial^L_{i,X}$ on $\cR_X$. In Lemmas \ref{lem:partialLXFi}, \ref{lem:pLiXFn}, \ref{lem:pLiXadFu}, \ref{lem:pLiXadFFn} we calculate the required values of $\partial^L_{i,X}$ explicitly.
  \subsection{Star products}\label{sec:star-def}
 We will extensively use the notion of star product algebras and their properties from \cite{a-KY20}.
For an $\N$-graded $\field$-algebra $A=\bigoplus_{j\in \N} A_j$ and $m\in \N$, denote $A_{<m}=\bigoplus_{j=0}^{m-1}A_j$ and $A_{\le m}=\bigoplus_{j=0}^{m}A_j$.
\begin{defi}\label{def:star-prod} Assume that $A=\bigoplus_{j\in \N} A_j$ is an $\N$-graded $\field$-algebra. 
A star product on $A$ is an associative bilinear operation $* : A \times A \to A$,
such that
\[
  a * b - ab \in A_{< m+ n} \qquad \mbox{for all $a \in A_m$, $b \in A_n$.}
\]
A star product $\ast$ on $A$ is called $0$-equivariant
if
\[
a * h = ah \quad \mbox{and} \quad h * a= h a \quad \mbox{for all} \quad h \in A_0, a \in A. 
\]
\end{defi}
In this setting $(A,\ast)$ is a filtered algebra with $\cF_m (A) := A_{\leq m}$ and
\[
\gr (A, *) \cong A.
\]
If $A$ is generated in degrees 0 and 1, then $(A,*)$ is generated by $\cF_1 (A)$ and we have 
the following properties (see \cite[Lemma 5.2.(ii)]{a-KY20}):
\begin{lem} 
\label{lem:sprod-isom}
Let $A$ be an $\N$-graded $\field$-algebra generated in degrees 0 and 1. 
\begin{enumerate}
\item[(i)] Every 0-equivariant star product on $A$ is uniquely determined by the $\field$-linear map $\mu^L:A_1\to \End_{\field}(A)$, 
\[
\mu^L_f(a) = f * a  - f a \qquad \mbox{for all $f \in A_1, a \in A$}.
\]
\item[(ii)] If $U$ is a graded subalgebra of $A$ such that $A_0 U = U A_0 = A$, then every 0-equivariant star product on $A$ is uniquely determined by the $\field$-linear map 
$\mu^L : U_1 \to \Hom_{\field}(U,A)$, 
\[
   \mu^L_f(u) = f*u - fu \qquad \mbox{for all $f \in U_1, u \in U$.}
\]
\end{enumerate}
\end{lem}
\subsection{The pull-back of the algebra structure on $\cB_\bc$}\label{sec:Bc-star}
We can use the Letzter map $\psi:\cB_\bc\rightarrow \cA$ to define a new algebra structure $\ast$ on $\cA$ by
\begin{align*}
  a\ast b = \psi\big(\psi^{-1}(a)\psi^{-1}(b)\big) \qquad \mbox{for all $a,b\in \cA$.}
\end{align*}  
By Lemma \ref{lem:psi-props}.(4) the binary operation $\ast$ defines a star product on $\cA$, and by Lemma \ref{lem:psi-props}.(2) this star product is $0$-equivariant. 
Hence Lemma \ref{lem:sprod-isom} implies that the star product $\ast$ on $\cA$ is uniquely determined by the family of maps
$\mu^L_f\in \Hom_{\field}(\cR_X,\cA)$ for all $f\in (\cR_X)_1=\cR_X\cap \cA_1$ defined by
\begin{align*}
  \mu_f^L(u) = f\ast u - fu \qquad \mbox{for all $u\in \cR_X$.}
\end{align*}  
As $\mu^L_{\ad_r(h)(f)}(u)=S(h_{(1)})\mu_f^L(h_{(2)}u)$ for all $f\in (\cR_X)_1$, $h\in \cH$ and $u\in \cR_X$, the $0$-equivariant star product $\ast$ on $\cA$ is uniquely determined by the maps $\mu_{F_i}^L$ for $i\in I\setminus X$. In the following Lemma we determine these maps.
\begin{lem}\label{lem:Fiastu}
  For all $i\in I\setminus X$ and all $u\in \cR_X$ we have
  \begin{align}\label{eq:Fiastu}
     F_i\ast u = F_i u - c_i \frac{q^{-(\alpha_i,\Theta(\alpha_i))}}{q_i-q_i^{-1}}K_{-\alpha_i-\Theta(\alpha_i)} T_{w_X} \circ \partial_{\tau(i)}^L\circ T_{w_X}^{-1}(u).
  \end{align}  
\end{lem}
\begin{proof}
  Write $u=\psi(b)$ for some $b\in \cB_\bc$. Using parts (1) and (3) of Lemma \ref{lem:psi-props} we calculate
  \begin{align*}
    F_i\ast u &= \psi(B_i b)\\
    &= \psi\big((F_i - c_i T_{w_X}(E_{\tau(i)})K_i^{-1})\psi(b) \big)\\
    &= F_i u - c_i q^{-(\alpha_i,\Theta(\alpha_i))} \psi\big(K_i^{-1}[T_{w_X}(E_{\tau(i)}),u] \big)\\
    &= F_i u - c_i  q^{-(\alpha_i,\Theta(\alpha_i))} T_{w_X}\bigg(\psi\big(K_{w_X(\alpha_i)}^{-1}[E_{\tau(i)}, T_{w_X}^{-1}(u)] \big)\bigg).
  \end{align*}
  Using Equation \eqref{eq:EyyE} and the definition of the projection $\psi$ we obtain
  \begin{align*}
    F_i\ast u &= F_i u - c_i \frac{q^{-(\alpha_i,\Theta(\alpha_i))}}{q_i-q_i^{-1}} T_{w_X} \big(K_{w_X(\alpha_i)}^{-1} K_{\tau(i)} \partial^L_{\tau(i)}(T_{w_X}^{-1}(u)) \big)\\
    &= F_i u - c_i \frac{q^{-(\alpha_i,\Theta(\alpha_i))}}{q_i-q_i^{-1}} K_{-\alpha_i - \Theta(\alpha_i)} \,T_{w_X} \big( \partial^L_{\tau(i)}( T_{w_X}^{-1}(u))\big)
  \end{align*}
  which gives the desired formula.
\end{proof}
\subsection{Twisted skew-derivations}\label{sec:twisted-skew}
Recall Corollary \ref{cor:RTwX} and the subalgebra $G_X^+$ defined in Section \ref{sec:setting}. The decomposition \eqref{eq:MRU-} implies that $T_{w_X}:U^-\rightarrow G_X^+\ot \cR_X$ is an algebra isomorphism. 
To shorten notation define twisted skew-derivations $\partial^L_{i,X}, \partial^R_{i,X}:G_X^+\ot\cR_X \rightarrow G_X^+\ot \cR_X$ by
\begin{align*}
  \partial^{L}_{i,X}=T_{w_X} \circ \partial^{L}_i \circ T_{w_X}^{-1}, \qquad
  \partial^{R}_{i,X}=T_{w_X} \circ \partial^{R}_i \circ T_{w_X}^{-1}
\end{align*}
for all $i\in I$.
The skew-derivation properties \eqref{eq:partialL}, \eqref{eq:partialR} of $\partial^L_i$ and $\partial^R_i$ imply that
\begin{align}
  \partial_{i,X}^L(fg)&=\partial_{i,X}^L(f)g +q^{(w_X(\alpha_i),\mu)}f \partial_{i,X}^L(g) \label{eq:partialLX}\\
  \partial_{i,X}^R(fg)&=q^{(w_X(\alpha_i),\nu)}\partial_{i,X}^R(f)g +f \partial_{i,X}^R(g) \label{eq:partialRX}
  \end{align}
for all $f\in (G_X^+\ot \cR_X)_{-\mu}$, $g\in (G_X^+\ot \cR_X)_{-\nu}$. To understand the star product better, we need to know the action of the twisted skew-derivation $\partial^L_{i,X}$ on the elements of $\cR_X$.
\begin{lem}\label{lem:partialLXFi}
  For any $i\in I \setminus X$ we have
  \begin{align}\label{eq:partialLXFi}
      \partial^L_{i,X}(F_i) = q^{(\alpha_i,\alpha_i-w_X(\alpha_i))} K_{\alpha_i-w_X(\alpha_i)} \partial_i^R(T_{w_X}(E_i)).
  \end{align}  
\end{lem}
\begin{proof}
  We apply $T_{w_X}^{-1}$ to Equation \eqref{eq:FxxF} for $x=T_{w_X}(E_i)$ and compare with Equation \eqref{eq:EyyE} to obtain
  \begin{align*}
    &\frac{T_{w_X}^{-1}(\partial_i^R(T_{w_X}(E_i))K_i) - T_{w_X}^{-1}(K_i^{-1}\partial^L_i(T_{w_X}(E_i)))}{q_i-q_i^{-1}}\\ 
    &\qquad\qquad=[E_i, T_{w_X}^{-1}(F_i)]   =\frac{K_i\partial^L_i(T_{w_X}^{-1}(F_i))-\partial_i^R(T_{w_X}^{-1}(F_i))K_i^{-1}}{q_i-q_i^{-1}}
    \cdot
  \end{align*}
  Comparing the summands involving $K_i$ we obtain
  \begin{align*}
     T_{w_X}^{-1}(\partial_i^R(T_{w_X}(E_i))K_i)=K_i\partial^L_i(T_{w_X}^{-1}(F_i))
  \end{align*}
  and hence
  \begin{align*}
    T_{w_X}\big(\partial_i^L(T_{w_X}^{-1}(F_i)) \big)&= K_{w_X(\alpha_i)}^{-1}\partial_i^R(T_{w_X}(E_i))K_i\\
    &=q^{(\alpha_i,\alpha_i-w_X(\alpha_i))} K_{\alpha_i-w_X(\alpha_i)}\partial_i^R(T_{w_X}(E_i)) 
  \end{align*}
  which confirms the statement of the lemma.
\end{proof}
For simplicity define 
\[
Z_i=\partial_i^R(T_{w_X}(E_i))\in U^+_{w_X(\alpha_i)-\alpha_i} \quad \mbox{for any} \quad i\in I\setminus X.
\]
As $w_X(\alpha_i)-\alpha_i\in Q_X^+$ we get $Z_i\in \cM_X^+$ and hence we have
\begin{align}\label{eq:ZiFj}
  [Z_i,F_j]=0 \qquad \mbox{for all $i,j\in I \setminus X$.}
\end{align}
Recall the non-symmetric quantum integer $(n)_p$ defined by $(n)_p=\sum_{j=0}^{n-1}p^j$ for any $n\in \N$, $p\in \field$.
\begin{lem}\label{lem:pLiXFn}
  For any $i\in I\setminus X$ and $n\in \N$ we have
  \begin{align}\label{eq:pLiXFn}
\partial_{i,X}^L(F_i^n)=q^{(\alpha_i,\alpha_i-w_X(\alpha_i))}(n)_{q_i^2}K_{\alpha_i-w_X(\alpha_i)} Z_i F_i^{n-1}.
  \end{align}
\end{lem}
\begin{proof}
  For $n=1$ the relation \eqref{eq:pLiXFn} holds by Lemma \ref{lem:partialLXFi}. We proceed by induction on $n$. For $n>1$ we calculate
  \begin{align*}
    &\partial_{i,X}^L(F_i^n)
    \stackrel{\eqref{eq:partialLX}}{=}\partial_{i,X}^L(F_i)F_i^{n-1} + q^{(w_X(\alpha_i),\alpha_i)} F_i\partial_{i,X}^L(F_i^{n-1})\\
    &\stackrel{\eqref{eq:partialLXFi}}{=}q^{(\alpha_i,\alpha_i-w_X(\alpha_i))} K_{\alpha_i-w_X(\alpha_i)}Z_i F_i^{n-1} +  F_i q^{(\alpha_i,\alpha_i)} (n-1)_{q_i^2} K_{\alpha_i-w_X(\alpha_i)} Z_i F_i^{n-2}\\
    &\stackrel{\phantom{\eqref{eq:partialLXFi}}}{=} q^{(\alpha_i,\alpha_i-w_X(\alpha_i))} (n)_{q_i^2} K_{\alpha_i-w_X(\alpha_i)} Z_i F_i^{n-1}
  \end{align*}
  where we also used the fact that $Z_i$ and $F_i$ commute by \eqref{eq:ZiFj}.
\end{proof}  
We will need similar relations for $\ad_r(F_j)(F_i^n)$ which lies in $\cR_X$ for $j\in X$. 
\begin{lem}\label{lem:pLiXadFu}
  For any $u\in \cR_X$, $j\in X$ and $i\in I\setminus X$ we have
  \begin{align} \label{eq:partial-adrFj}
    \partial_{i,X}^L(\ad_r(F_j)(u)) =   - \frac{1}{q_j-q_j^{-1}} \partial_{\tau(j),X}^R \circ \partial_{i,X}^L(u).
  \end{align}  
\end{lem}  
\begin{proof}
  For $u\in (\cR_X)_{-\mu}$ and $j\in X$ we calculate
  \begin{align}
    T_{w_X}^{-1}(\ad_r(F_j)(u))& \stackrel{\phantom{\eqref{eq:T_i-equivalence}}}{=}T_{w_X}^{-1}\big(u F_j - q^{-(\alpha_j,\mu)}F_j u\big)\nonumber\\
      &\stackrel{\phantom{\eqref{eq:T_i-equivalence}}}{=} - T_{w_X}^{-1}(u) E_{\tau(j)} K_{\tau(j)} + q^{-(\alpha_j,\mu)} E_{\tau(j)} K_{\tau(j)} T_{w_X}^{-1}(u)\nonumber\\
        &\stackrel{\phantom{\eqref{eq:T_i-equivalence}}}{=} [E_{\tau(j)}, T_{w_X}^{-1}(u)] K_{\tau(j)}\nonumber\\
        &\stackrel{\eqref{eq:T_i-equivalence}}{=} - \frac{1}{q_j-q_j^{-1}} \partial^R_{\tau(j)}\circ T_{w_X}^{-1}(u).\label{eq:partialRtauj}
  \end{align}
  Hence we obtain
  \begin{align*}
    \partial_{i,X}^L(\ad_r(F_j)(u)) &=  - \frac{1}{q_j-q_j^{-1}} T_{w_X}\circ \partial^L_i \circ \partial^R_{\tau(j)}\circ T_{w_X}^{-1}(u)\\
    &=  - \frac{1}{q_j-q_j^{-1}} T_{w_X}\circ \partial^R_{\tau(j)} \circ \partial^L_i \circ T_{w_X}^{-1}(u)\\
    &=  - \frac{1}{q_j-q_j^{-1}} \partial_{\tau(j),X}^R \circ \partial_{i,X}^L(u)
  \end{align*}
  which is the desired formula.
\end{proof}  
By Equation \eqref{eq:partial-adrFj} we will need to understand the action of $\partial_{j,X}^R$ on $G_X^+\ot \cR_X$. Equation \eqref{eq:partialRtauj} implies that
  \begin{align}\label{eq:partialRj-ad}
    \partial_{j,X}^R(u)=-(q_j-q_j^{-1}) \ad_r(F_{\tau(j)})(u) \qquad \mbox{for all $u\in \cR_X$, $j\in X$.}
  \end{align}
Moreover, by Lemma \ref{lem:partialLXFi} we have
  \begin{align}\label{eq:TwXmKZ}
    T_{w_X}^{-1}(K_{\alpha_i-w_X(\alpha_i)}Z_i) &= q^{(\alpha_i, w_X(\alpha_i)-\alpha_i)} \partial_i^L\circ T_{w_X}^{-1}(F_i)
  \end{align}
  and hence
  \begin{align*}
    \partial_{\tau(j),X}^R(K_{\alpha_i-w_X(\alpha_i)}Z_i)=
    q^{(\alpha_i, w_X(\alpha_i)-\alpha_i)} \partial_{\tau(j),X}^R \circ \partial_{i,X}^L (F_i)
  \end{align*}
  which implies that $\partial_{\tau(j),X}^R(K_{\alpha_i-w_X(\alpha_i)}Z_i)\in  (G_X^+)_{w_X(\alpha_i)-\alpha_i-\alpha_j}$.
  \begin{lem}\label{lem:pLiXadFFn}
    For any $i\in I\setminus X$, $j\in X$ and $1\le n\in \N$ we have
  \begin{align}
    \partial^L_{i,X}\big(\ad_r(F_j)&(F_i^n)\big)\nonumber\\
    =& \frac{q^{(\alpha_i+\alpha_j,\alpha_i+\alpha_j-w_X(\alpha_i))-n(\alpha_i,\alpha_j)}}{q_j-q_j^{-1}} (n)_{q_i^2} K_{\alpha_i+\alpha_j-w_X(\alpha_i)} \partial^R_j(Z_i) F_i^{n-1}\label{eq:partialLiX-adrFjFin}\\
    &+ q^{(\alpha_i,\alpha_i-w_X(\alpha_i))}(n)_{q_i^2} K_{\alpha_i-w_X(\alpha_i)}Z_i \ad_r(F_j)(F_i^{n-1}) \nonumber
  \end{align}
  \end{lem}
  \begin{proof}
    By equation \eqref{eq:EyyE} we have
    \begin{align*}
      [E_{\tau(j)}&,T_{w_X}^{-1}(K_{\alpha_i-w_X(\alpha_i)}Z_i)]\\ &=\frac{K_{\tau(j)}\partial^L_{\tau(j)}(T_{w_X}^{-1}(K_{\alpha_i-w_X(\alpha_i)}Z_i))-\partial^R_{\tau(j)}(T_{w_X}^{-1}(K_{\alpha_i-w_X(\alpha_i)}Z_i))K_{\tau(j)}^{-1}}{q_j-q_j^{-1}}.
    \end{align*}
    Applying $T_{w_X}$ to this equation we obtain
    \begin{align*} &\frac{K_j^{-1}\partial_{\tau(j),X}^L(K_{\alpha_i-w_X(\alpha_i)}Z_i)-\partial_{\tau(j),X}^R(K_{\alpha_i-w_X(\alpha_i)}Z_i)K_{j}}{q_j-q_j^{-1}}\\
     &=-[F_j K_j, K_{\alpha_i-w_X(\alpha_i)}Z_i]\\
      &= - K_{\alpha_i-w_X(\alpha_i)}[F_j,Z_i]K_j\\
      &= - K_{\alpha_i-w_X(\alpha_i)}\, \frac{K_j^{-1}\partial_j^L(Z_i) - \partial_j^R(Z_i)K_j}{q_j-q_j^{-1}}\, K_j 
    \end{align*}
    Using the fact that $\partial_{\tau(j),X}^R(K_{\alpha_i-w_X(\alpha_i)}Z_i) \in (G_X^+)_{w_X(\alpha_i)-\alpha_i-\alpha_j}$ noted above, and collecting terms in $U^+K_{\alpha_i+2\alpha_j-w_X(\alpha_i)}$ we obtain
    \begin{align*}
      \partial_{\tau(j),X}^R(K_{\alpha_i-w_X(\alpha_i)} Z_i)
      &=-q^{-(\alpha_j,w_X(\alpha_i)-\alpha_i-\alpha_j)} K_{\alpha_i-w_X(\alpha_i)+\alpha_j}\partial^R_j(Z_i).
    \end{align*}
We can use the above formula and the skew-derivation property \eqref{eq:partialRX} to calculate    
  \begin{align*}
    \partial^L_{i,X}\big(&\ad_r(F_j)(F_i^n)\big)
    \stackrel{\eqref{eq:partial-adrFj}}{=}  \frac{-1}{q_j-q_j^{-1}} \partial_{\tau(j),X}^R \circ \partial_{i,X}^L(F_i^n)\\
    \stackrel{\eqref{eq:pLiXFn}}{=}& \frac{-1}{q_j-q_j^{-1}} \partial_{\tau(j),X}^R
    \Big(q^{(\alpha_i,\alpha_i-w_X(\alpha_i))}(n)_{q_i^2}K_{\alpha_i-w_X(\alpha_i)} Z_i F_i^{n-1}  \Big)\\
     \stackrel{\phantom{\eqref{eq:partial-adrFj}}}{=}&   \frac{q^{(\alpha_i+\alpha_j,\alpha_i+\alpha_j-w_X(\alpha_i))-n(\alpha_i,\alpha_j)}}{q_j-q_j^{-1}} (n)_{q_i^2} K_{\alpha_i+\alpha_j-w_X(\alpha_i)} \partial_j^R(Z_i) F_i^{n-1}\\
    &+ q^{(\alpha_i,\alpha_i-w_X(\alpha_i))}(n)_{q_i^2} K_{\alpha_i-w_X(\alpha_i)}Z_i \ad_r(F_j)(F_i^{n-1})
  \end{align*}
  where we also used \eqref{eq:partialRj-ad} in the last step.
\end{proof}
\section{Continuous $q$-Hermite polynomials and deformed Chebyshev polynomials of the second kind}
\label{sec:orth-polyn}
We recall the family of bivariate continuous $q$-Hermite polynomials introduced in \cite{a-CKY20}. Here we will need rescaled versions of these polynomials with coefficients in the algebra $\cM_X^+$. To express the defining relations of $\cB_\bc$ we will consider quantum Serre combinations of bivariate $q$-Hermite polynomials which we hence discuss in some detail. We also introduce a new class of deformed Chebyshev polynomials of the second kind in preparation for the quantum Serre relations in the subtle case $i\in I\setminus X$, $j\in X$ discussed in Section \ref{sec:caseII}.
\subsection{Bivariate continuous $q$-Hermite polynomials}\label{sec:biv-q-Herm}
Recall the univariate continuous $q$-Hermite polynomials $H_m(x;q)$ defined recursively for all $m\in \Z$ by $H_{-m}(x;q)=0$ for $m<0$, $H_{0}(x;q)=1$, and
\begin{align}\label{eq:Hm-recursion}
  2x H_{m}(x;q) = H_{m+1}(x;q) + (1-q^m) H_{m-1}(x;q),
\end{align}  
see \cite[14.26]{b-KLS10}. The following family of bivariate continuous $q$-Hermite polynomials $H_{m,n}(x,y;q,r)$ was introduced in \cite{a-CKY20}.
\begin{defi}
  Let $r\in \field$. The bivariate continuous $q$-Hermite polynomials $H_{m,n}(x,y;q,r)$ are defined for all $m,n\in \Z$ by $H_{0,n}(x,y;q,r)=H_n(y;q)$ and the recursion
  \begin{align}\label{eq:Hmn-recursion}
    2x H_{m,n}(x,y;q,r) = H_{m+1,n}(x,y;q,r) &+ (1-q^m)H_{m-1,n}(x,y;q,r)\\
                                          & + q^m(1-q^n)r H_{m,n-1}(x,y;q,r) \nonumber
  \end{align}
where we set $H_{-m,n}(x,y;q,r) =0$ for $m >0, n \in \Z$ and $H_{m,-n}(x,y;q,r)=0$  $m \in \Z, n >0$. 
\end{defi}
It was shown in \cite{a-CKY20} that the bivariate continuous $q$-Hermite polynomials form a family of orthogonal polynomials with many desirable properties. In particular, they satisfy the symmetry condition
\begin{align*}
  H_{m,n}(x,y;q,r) = H_{n,m}(y,x;q,r) \qquad \mbox{for all $m,n\in \Z$}.
\end{align*}
Moreover, by \cite[(3.4)]{a-CKY20} the bivariate $q$-Hermite polynomials $H_{m,n}(x,y;q,r)$ can be expressed in terms of the univariate continuous $q$-Hermite polynomials by the formula
\begin{align}\label{eq:bi-in-uni}
  H_{m,n}(x,y;q,r)=\sum_{k=0}^{\min(m,n)}\frac{(-1)^kq^{{k\choose 2}}(q;q)_m (q;q)_n r^k}{(q;q)_{m-k}(q;q)_{n-k}(q;q)_k} H_{m-k}(x;q) H_{n-k}(y;q) 
\end{align}
where $(t;q)_k=\prod_{j=0}^{k-1}(1-tq^j)$ denotes the $q$-Pochhammer symbol. Note that the coefficient inside the sum \eqref{eq:bi-in-uni} is a polynomial in $q$.

In the following we will encounter rescaled versions of the continuous $q$-Hermite polynomials with coefficients in an associative algebra $\cM$ over $\field$. Let $\cM[x]=\cM\ot_\field \field[x]$ and $\cM[x,y]=\cM\ot_\field \field[x,y]$ be the algebras of polynomials in one or two variables with coefficients in $\cM$.
\begin{lem}\label{lem:wmn}
  Let $b^2\in\cM\setminus \{0\}$ and let $b$ be a formal square root of $b^2$. \begin{enumerate}
    \item For all $m,n\in\N$ the expressions
  \begin{align}\label{eq:wm-wmn-def}
    w_m(x)=(b/2)^m H_m(\frac{x}{b};q^2) \quad \mbox{and} \quad
    w_{m,n}(x,y)= (b/2)^{m+n} H_{m,n}(\frac{x}{b},\frac{y}{b};q^2,r)
  \end{align}
  define polynomials in $\cM[x]$ and $\cM[x,y]$, respectively.
\item The polynomials $w_m(x)$ for $m\in\N$ are uniquely determined by $w_0(x)=1$, $w_1(x)=x$ and the recursion
  \begin{align}\label{eq:wm-recursion}
     x w_m(x) = w_{m+1}(x) +\frac{b^2}{4} (1-q^{2m})w_{m-1}(x).
  \end{align}  
\item The polynomials $w_{m,n}(x)$ are uniquely determined by $w_{0,n}(x,y)=w_n(y)$ for all $n\in \N$ and by the recursion
  \begin{align}\label{eq:wmn-recursion}
    x w_{m,n}(x,y)= w_{m+1,n}(x,y) &+\frac{b^2}{4} (1-q^{2m})w_{m-1,n}(x,y)\\
                     &+ \frac{b^2}{4} q^{2m}(1-q^{2n})r w_{m,n-1}(x,y) \nonumber
  \end{align}
  where we set $w_{-1,n}(x,y)=w_{m,-1}(x,y)=0$. 
\item For all $m,n\in \N$ the relation
  \begin{align}\label{eq:wmn}
    w_{m,n}(x,y)=\sum_{k=0}^{\min(m,n)}r^k q^{k(m+n)}\eta(k) \begin{bmatrix}m\\k\end{bmatrix}_q \begin{bmatrix}n\\k \end{bmatrix}_q[k]^!_q  w_{m-k}(x) w_{n-k}(y) 
    \end{align}
  holds with $\eta(k)=(q-q^{-1})^k q^{-k(k+1)/2} (b^2/4)^k$.
  \end{enumerate}
\end{lem}
\begin{proof}
  The recursion \eqref{eq:Hmn-recursion} implies that $H_{m,n}(x,y;q^2,r)$ is a polynomial with leading term $x^my^n$, and that if a monomial $x^iy^j$ appears in $H_{m,n}(x,y;q^2,r)$ with nonzero coefficient then $i+j\equiv m+n$ mod $2$. This implies that \eqref{eq:wm-wmn-def} indeed defines polynomials in $\cM[x]$ and $\cM[x,y]$. This proves (1). Properties (2) and (3) are obtained by replacing $x$ and $y$ by $x/b$ and $y/b$, respectively, and $q$ by $q^2$ in the recursions \eqref{eq:Hm-recursion} and \eqref{eq:Hmn-recursion}. Property (4) follows from \eqref{eq:bi-in-uni} and the relations
\begin{align*}
(q^2,q^2)_k&=(-1)^k q^{k(k+1)/2}(q{-}q^{-1})^k[k]_q^!, & \frac{(q^2;q^2)_n}{(q^2;q^2)_m (q^2;q^2)_{n-m}}&= q^{m(n-m)}\begin{bmatrix}n \\ m \end{bmatrix}_{q}
\end{align*}
  for $n\ge m$.
\end{proof}
\subsection{A relation between $q$-Hermite polynomials for $q^2$ and $q^{-2}$}
In the following we fix $b^2\in\cM$ and we let $w_m(x)\in \cM[x]$ be the polynomial defined by \eqref{eq:wm-wmn-def}. Moreover, set
\begin{align}\label{eq:vm-def}
  v_m(x)=\sum_{k=0}^{\lfloor m/2 \rfloor} (-1)^k q^k \eta(k) \frac{[m]^!_q}{[m{-}2k]^!_q [k]^!_q} w_{m-2k}(x) \in \cM[x]. 
\end{align}
where as before $\eta(k)=(q-q^{-1})^k q^{-k(k+1)/2} (b^2/4)^k$. In the following we set $w_k(x)=0$ for all $k< 0$, and hence in the definition \eqref{eq:vm-def} of $v_m(x)$ we may take the sum over all non-negative integers $k$.
\begin{lem}\label{lem:vm-recursion}
  The polynomials $v_m(x)$ satisfy the recursion
  \begin{align}\label{eq:vm-recursion}
     x v_m(x)= v_{m+1}(x) + \frac{b^2}{4}(1-q^{-2m}) v_{m-1}(x)
  \end{align}
  with $v_0(x)=1$ and $v_1(x)=x$. Hence we have
  \begin{align}\label{eq:vm-def2}
     v_m(x) = (b/2)^m H_m(\frac{x}{b},q^{-2})
  \end{align}
  for all $m\in \N$.
\end{lem}  
\begin{proof}
  Using the definition \eqref{eq:vm-def} of $v_m(x)$ and the recursion \eqref{eq:wm-recursion} we calculate
  \begin{align*}
    &xv_m(x) - \frac{b^2}{4}(1-q^{-2m})v_{m-1}(x)\\
    =&\sum_{k=0}^{\lfloor m/2 \rfloor} (-1)^k q^k \eta(k) \frac{[m]^!_q}{[m{-}2k]^!_q [k]^!_q} \Big(w_{m-2k+1}(x) + \frac{b^2}{4} (1{-}q^{2(m-2k)})w_{m-2k-1}(x)\Big)\\
    &-\frac{b^2}{4}(1-q^{-2m})\sum_{k=0}^{\lfloor (m-1)/2 \rfloor} (-1)^k q^k \eta(k) \frac{[m-1]^!_q}{[m{-}1{-}2k]^!_q [k]^!_q}w_{m-1-2k}(x)\\
    =& w_{m+1}(x)
    +\sum_{k\ge 1} (-1)^k q^k \eta(k) \frac{[m]^!_q}{[m{+}1{-}2k]^!_q [k]^!_q} \cdot\\
    & \qquad \qquad\qquad \qquad \cdot\Big([m{+}1{-}2k]_q + [k]_q(q^{m-k+1} + q^{-m+k-1}) \Big)w_{m+1-2k}(x)\\
    =&v_{m+1}(x).
  \end{align*}
  This confirms the recursion \eqref{eq:vm-recursion}. The second statement holds by parts (1) and (2) of Lemma \ref{lem:wmn} with $q$ replaced by $q^{-1}$.
\end{proof}
\subsection{Serre combinations of bivariate continuous $q$-Hermite polynomials}\label{sec:wmn-Serre-combi}
The following identity is the crucial ingredient needed in the following sections to express the quantum Serre relations for quantum symmetric pairs in terms of univariate continuous $q$-Hermite polynomials.
\begin{prop}\label{prop:Serre-bi-uni}
  Let $a\in -\N$ and assume that $r=q^a$. Then the polynomials $w_m(x), v_n(x)\in \cM[x]$ and $w_{m,n}(x,y)\in \cM[x,y]$ defined by \eqref{eq:wm-wmn-def} and \eqref{eq:vm-def} satisfy the relation
  \begin{align*}
    \sum_{n=0}^{1-a} (-1)^n \begin{bmatrix} 1-a\\ n \end{bmatrix}_q w_{1-a-n,n}(x,y) &= \sum_{n=0}^{1-a} (-1)^n \begin{bmatrix} 1-a\\ n \end{bmatrix}_q w_{1-a-n}(x)\,v_n(y) \\
    &= \sum_{n=0}^{1-a} (-1)^n \begin{bmatrix} 1-a\\ n \end{bmatrix}_q v_{1-a-n}(x)\,w_{n}(y).
  \end{align*}  
\end{prop}
\begin{proof}
  For $r=q^a$ and any $b^2\in \cM$ we use Equation \eqref{eq:wmn} to calculate
  \begin{align*}
    &\sum_{n=0}^{1-a} (-1)^n \begin{bmatrix} 1-a\\ n \end{bmatrix}_q w_{1-a-n,n}(x,y)\\
    &= \sum_{n=0}^{1-a} (-1)^n \begin{bmatrix} 1{-}a\\ n \end{bmatrix}_q
   \sum_{k\ge 0} q^{k}\eta(k)\begin{bmatrix} 1{-}a{-}n\\ k \end{bmatrix}_q \begin{bmatrix} n\\ k \end{bmatrix}_q [k]^!_q w_{1-a-n-k}(x)w_{n-k}(y)
  \end{align*}
  where again $w_n(x)=0$ for $n<0$. Setting $\ell=n+k$ and $m=k$ we obtain
   \begin{align*}
    &\sum_{n=0}^{1-a} (-1)^n \begin{bmatrix} 1-a\\ n \end{bmatrix}_q w_{1-a-n,n}(x,y)\\
     &= \sum_{\ell=0}^{1-a} (-1)^\ell \begin{bmatrix} 1{-}a\\ \ell \end{bmatrix}_q  \sum_{m\ge 0} (-1)^mq^m \eta(m)\frac{[\ell]^!_q}{[\ell{-}2m]^!_q [m]^!_q} w_{1-a-\ell}(x) w_{\ell-2m}(y)\\
     &=\sum_{\ell=0}^{1-a} (-1)^\ell \begin{bmatrix} 1{-}a\\ \ell \end{bmatrix}_q w_{1-a-\ell}(x) v_{\ell}(y)
   \end{align*}
   which proves the first identity of the proposition. On the other hand, setting $\ell=1-a-n+k$ and $m=k$ we obtain
    \begin{align*}
    &\sum_{n=0}^{1-a} (-1)^n \begin{bmatrix} 1-a\\ n \end{bmatrix}_q w_{1-a-n,n}(x,y)\\
      &=(-1)^{1-a} \sum_{\ell=0}^{1-a} (-1)^\ell \begin{bmatrix} 1{-}a\\ \ell \end{bmatrix}_q   \sum_{m\ge 0} (-1)^mq^m \eta(m)\frac{[\ell]^!_q}{[\ell{-}2m]^!_q [m]^!_q} w_{\ell-2m}(x) w_{1-a-\ell}(y)\\
      &=(-1)^{1-a}\sum_{\ell=0}^{1-a} (-1)^\ell \begin{bmatrix} 1{-}a\\ \ell \end{bmatrix}_q v_{\ell}(x) w_{1-a-\ell}(y)
      =\sum_{\ell=0}^{1-a} (-1)^\ell \begin{bmatrix} 1{-}a\\ \ell \end{bmatrix}_q v_{1-a-\ell}(x) w_{\ell}(y)
    \end{align*}
    which proves the second identity.
\end{proof}
\subsection{Bar invariance of deformed quantum Serre polynomials}
Assume that the $\field$-algebra $\cM$ has a bar involution
\begin{align*}
  \barM:\cM \rightarrow \cM, \qquad m\mapsto \overline{m}^\cM
\end{align*}
which is a $k$-algebra homomorphism such that $\overline{q}^\cM=q^{-1}$. We extend $\barM$ to a bar involution on $\cM[x]$ and $\cM[x,y]$ by application of $\barM$ to the coefficients of any polynomial. Again, we consider the polynomials $w_m(x), v_m(x)$ and $w_{m,n}(x,y)$ formed with respect to a fixed element $b^2\in \cM\setminus \{0\}$. Define $(b')^2=\overline{b^2}^\cM$ and let $w_m'(x), v_m'(x)$ and $w_{m,n}'(x,y)$ be the corresponding polynomials formed with respect to the element $(b')^2$.
\begin{lem}\label{lem:ow=v}
  Retain the setting of Proposition \ref{prop:Serre-bi-uni} and set  $(b')^2=\overline{b^2}^\cM$. Then the relations
  \begin{align}\label{eq:ow=v}
     \overline{w_n(x)}^\cM=v_n'(x), \qquad \overline{v_n(x)}^\cM=w_n'(x)
  \end{align}
 hold in $\cM[x]$ for all $n\in \N$. 
\end{lem}
\begin{proof}
  By Lemma \ref{lem:wmn}.(2) and Lemma \ref{lem:vm-recursion} we have $\overline{w_0(x)}^\cM=1=v_0'(x)$ and $\overline{w_1(x)}^\cM=x=v_1'(x)$ and the recursions \eqref{eq:wm-recursion}, \eqref{eq:vm-recursion} inductively imply the first relation in \eqref{eq:ow=v}. The second relation in \eqref{eq:ow=v} is obtained analogously. 
\end{proof}
The above lemma and Proposition \ref{prop:Serre-bi-uni} allow us to describe the behavior of the deformed quantum Serre polynomial under the bar involution.
\begin{cor}\label{cor:def-qSerre-bar}
  Retain the setting of Proposition \ref{prop:Serre-bi-uni} and set  $(b')^2=\overline{b^2}^\cM$. Then the relation
  \begin{align*}
   \overline{\sum_{n=0}^{1-a}(-1)^n \begin{bmatrix} 1-a\\ n \end{bmatrix}_q w_{1-a-n,n}(x,y)}^\cM = \sum_{n=0}^{1-a}(-1)^n \begin{bmatrix} 1-a\\ n \end{bmatrix}_q w_{1-a-n,n}'(x,y)
  \end{align*}
 holds in $\cM[x,y]$.
\end{cor}  
\begin{proof}
  By Lemma \ref{lem:ow=v} and Proposition \ref{prop:Serre-bi-uni} we have
  \begin{align*}
     \overline{\sum_{n=0}^{1-a} (-1)^n \begin{bmatrix} 1-a\\ n \end{bmatrix}_q w_{1-a-n,n}(x,y)}^\cM
    &= \sum_{n=0}^{1-a} (-1)^n \begin{bmatrix} 1-a\\ n \end{bmatrix}_q \overline{w_{1-a-n}(x)}^\cM\,\overline{v_n(y)}^\cM\\
    &=\sum_{n=0}^{1-a} (-1)^n \begin{bmatrix} 1-a\\ n \end{bmatrix}_q v_{1-a-n}'(x)\,w_n'(y)\\
    &= \sum_{n=0}^{1-a} (-1)^n \begin{bmatrix} 1-a\\ n \end{bmatrix}_q w_{1-a-n,n}'(x,y)
  \end{align*}
  as desired.
\end{proof}  
\subsection{Deformed Chebyshev polynomials of the second kind}\label{sec:Chebyshev}
For any $r\in \field$ consider the sequence of polynomials given by the recursion 
formula
\begin{equation}
\label{eq:Du-recursion}
    \Dc_{n+1}(x; q, r) = 2x \Dc_{n}(x; q, r) - \frac{r^{-1}-q^{n+1}}{1 - q^{n+1}} \Dc_{n-1}(x; q, r), \quad n \geq 0 
\end{equation}
subject to the initial conditions 
\begin{equation}
\label{eq:Du-ic}
\Dc_0(x; q, r) =1, \quad \Dc_1(x; q, r) =2x.
\end{equation}
It is clear from the recursion and the initial conditions that in the case $r=1$ we recover the Chebyshev polynomials of the second kind:
\[
\Dc_n(x; q, 1) = U_n(x),
\]
see e.g. \cite[Sect. 5.7]{b-MOS} for background on the classical Chebyshev polynomials.
\begin{eg}\label{eg:Du-explicit}
  By direct computation one obtains
  \begin{align}
    \Dc_2(x;q,r)&= 4 x^2-  \frac{r^{-1}-q^2}{(1-q^2)}, \label{eq:Du2}\\
    \Dc_3(x;q,r)&=8x^3 - 2x\Big( \frac{r^{-1}-q^2}{1-q^2} +  \frac{r^{-1}-q^3}{1-q^3} \Big) .\label{eq:Du3}
  \end{align}  
\end{eg}  

By Favard's theorem, $\{\Dc_n(x)\}_{n=0}^\infty$ are a sequence of orthogonal polynomials with respect to a Borel measure 
when $q$ and $r$ are real and $|q|< 1 \leq r^{-1}$.
We will call them {\em{deformed Chebyshev polynomials of the second kind}}.
If $q$ and $r$ are viewed as indeterminates, then 
\[
(1-q^2) \ldots (1-q^n) \Dc_n(x; q, r) \in \Z[q,r^{-1}], \quad \mbox{for} \; \; n \geq 2.
\]

The rescaled polynomials 
\[
\widetilde{\Dc}_n(x; q, r) = r^{n/2} \Dc_n(r^{-1/2} x; q, r)
\]
satisfy the recursion 
\[
\widetilde{\Dc}_{n+1}(x; q, r) = 2x \widetilde{\Dc}_{n}(x; q, r) - \frac{1- r q^{n+1}}{1 - q^{n+1}} \widetilde{\Dc}_{n-1}(x; q, r), \quad n \geq 0 
\]
and the same initial conditions as $\Dc_{n}(x; q, r)$. In our recursion for the deformed Chebyshev 
polynomials we used $r^{-1}$ instead of $r$ because the above recursion naturally 
leads to expressions involving $q$-Pochhammer symbols.

Next we provide explicit formulas for two generating functions for the deformed Chebyshev polynomials. In an appropriate field extension of $\field=k(q)$ define 
\begin{equation}
\label{eq:xa12}
x_{1,2}=x\pm \sqrt{x^2-1}, \quad
a_{1,2}=x\pm \sqrt{x^2-r^{-1}}, \quad
\widetilde{a}_{1,2} = r^{-1}(x\pm \sqrt{x^2-r}),
\end{equation}
which satisfy the relations 
\begin{align*}
1-2xs+s^2&=(1- x_1s)(1-x_2s), 
\\
1-2xs+r^{-1}s^2&=(1-a_1s)(1-a_2s),\\ 
1-2r^{-1} xs+r^{-1}s^2&=(1-\widetilde{a}_1s)(1-\widetilde{a}_2s).
\end{align*}
Recall the definition of the basic hypergeometric function ${}_3\phi_2$, see \cite[1.10]{b-KLS10}, and define
\begin{align} \label{eq:def-eta}
  \eta\left(\begin{matrix} x \\s \end{matrix}\,;\,q\right)&=\frac{1}{1-2xs+r^{-1}s^2} {}_3\phi_2\left(\begin{array}{c} q,x_1s,x_2s \\qa_1s,qa_2s\end{array};q,q^2\right)\\
    &=\sum_{k=0}^\infty \frac{(x_1s;q)_k(x_2s;q)_k}{(a_1s;q)_{k+1}(a_2s;q)_{k+1}}\,q^{2k} \nonumber
\end{align}
As the basic hypergeometric function ${}_3\phi_2\left(\begin{array}{c} q,x_1s,x_2s \\qa_1s,qa_2s\end{array};q,q^2\right)$ is analytic at $s=0$, 
so is  $\eta\left(\begin{matrix} x \\s \end{matrix}\,;\,q\right)$. Analogously 
\[
\widetilde{\eta} \left(\begin{matrix} x \\s \end{matrix}\,;\,q\right)=\frac{1}{1-2xs+s^2} {}_3\phi_2\left(\begin{array}{c} q, q \widetilde{a}_1s, q \widetilde{a}_2 s 
\\qx_1s,qx_2s\end{array};q,q r\right)
\]
is analytic at $s=0$.
\begin{prop}
\label{prop:Du-genfn} We have
\begin{align*}
   &\eta\left(\begin{matrix} x \\s \end{matrix}\,;\,q, r\right)= \sum_{n=0}^\infty \Dc_n(x;q,r) \frac{s^n}{1-q^{n+2}} \quad \mbox{and} \\
   &\widetilde{\eta}\left(\begin{matrix} x \\s \end{matrix}\,;\,q, r\right)= \sum_{n=0}^\infty \frac{(q^2;q)_n}{(qr;q)_{n+1}} \widetilde{\Dc}_n(x;q,r) s^n 
   = \sum_{n=0}^\infty \frac{(q^2;q)_n}{(qr;q)_{n+1}} \Dc_n(r^{-1/2}x;q,r) (r^{1/2} s)^n.
\end{align*}
\end{prop}
The first generating function will play a key role in Sections \ref{sec:gf}--\ref{sec:PN}. In the special case $r=1$, the second one reduces to the 
standard generating function for the Chebyshev polynomials of the second kind
\[
\sum_{n=0}^\infty U_n(x) s^n = (1 - 2xs +s^2)^{-1}.
\]
\begin{proof} For simplicity of the notation we suppress the arguments $s$ and $q$ of the functions $\eta$ and $\widetilde{\eta}$. 
Consider the Taylor expansion 
\[
   \eta(x)= \sum_{n=0}^\infty f_n(x;q) s^n.
\]
in the variable $s$ with Taylor coefficients $f_n(x;q)$. We have
 \begin{align*}
    \eta(qs) =& \sum_{k=0}^\infty \frac{(qx_1s;q)_k(qx_2s;q)_k}{(qa_1s;q)_{k+1}(qa_2s;q)_{k+1}}\,q^{2k} \\
    =& q^{-2} \frac{(1-a_1s)(1-a_2s)}{(1-x_1s)(1-x_2s)} \sum_{k=1}^\infty
    \frac{(x_1s;q)_k(x_2s;q)_k}{(a_1s;q)_{k+1}(a_2s;q)_{k+1}}\,q^{2k}\\
    =& q^{-2} \frac{(1-a_1s)(1-a_2s)}{(1-x_1s)(1-x_2s)} \left(\eta(s)-\frac{1}{(1-a_1s)(1-a_2s)}\right)\\
    =& q^{-2} \frac{1-2xs+r^{-1}s^2}{1-2xs+s^2} \left(\eta(s)-\frac{1}{1-2xs+r^{-1}s^2}\right),
  \end{align*}
and thus,   
\begin{align*}
  q^2(1-2xs+s^2) \eta(qs) = (1-2xs+r^{-1}s^2) \eta(s)-1.
\end{align*}
Comparing the coefficients of $s^{n+1}$ in the Taylor series expansion of both sides at $s=0$ gives
\begin{align*}
  q^{n+3} f_{n+1}(x;q) -&2 q^{n+2}x f_{n}(x;q) + q^{n+1} f_{n-1}(x;q) \\
    &= f_{n+1}(x;q) - 2x f_{n}(x;q) + r^{-1}f_{n-1}(x;q) - \delta_{n,-1}.
\end{align*}
Hence, $f_n(x;q)$ are uniquely determined by the recursion 
\begin{align*}
    (1-q^{n+3}) f_{n+1}(x;q) = 2x(1-q^{n+2}) f_{n}(x;q) - (r^{-1}-q^{n+1})f_{n-1}(x;q) + \delta_{n,-1}
\end{align*}
with the initial condition $f_n(x;q)=0$ for $n<0$. Comparing this to the recursion \eqref{eq:Du-recursion} and initial conditions \eqref{eq:Du-ic} 
gives that 
\[
f_n(x;q) = \Dc_n(x; q, r)/(1-q^{n+2}) \quad \mbox{for} \quad n \geq 0
\]
which proves the first generating function identity. Similarly, for the second identity, one first shows that
\begin{align*}
    \widetilde{\eta}(qs) =  (qr)^{-1} \frac{1-2xs+s^2}{1-2 r^{-1} q xs+ r^{-1} q^2 s^2} \left(\widetilde{\eta}(s)-\frac{1}{1-2xs+s^2}\right)
\end{align*}
and then proceeds as above.
\end{proof}
\section{Generators and relations for $\cB_\bc$}\label{sec:GenRel}
We are now ready to deduce the quantum Serre relations \eqref{eq:S=C-intro} for $\cB_\bc$. The cases (I) and (III) from Section \ref{sec:prev} where $i,j\in I\setminus X$ are treated in Sections \ref{sec:caseI} and \ref{sec:caseIII}, respectively, and are rather straightforward generalizations of the calculations in the quasi-split setting in \cite{a-CKY20}. The bulk of this section, Sections \ref{sec:jinX}--\ref{sec:caseII}, is devoted to the subtle case (II) where $i\in I\setminus X$ and $j\in X$ and which does not exist in the quasi-split setting.
\subsection{$\cM_X^+$-valued orthogonal polynomials}
\label{sec:o-p}
For a suitable choice of $b^2\in \cM_X^+$ the recursions \eqref{eq:wm-recursion} and \eqref{eq:wmn-recursion} translate into the recursions given in A) and B) in Section \ref{sec:statement}. Indeed, recall that $\cZ_i=q_ic_i \partial_{\tau(i)}^R(T_{w_X(E_{\tau(i)})})$ for $i\in I\setminus X$ and set
\begin{align*}
b_i^2= \frac{4}{(q_i-q_i^{-1})^2}\cZ_i\in \cM_X^+.
\end{align*}
For $b^2=b_i^2$ the rescaled continuous $q$-Hermite polynomials $w_m(x)$ defined for $m\in \N$ by 
\begin{align}\label{eq:iwm-def}
  w_m(x) = (b_i/2)^m H_m(\frac{x}{b_i};q_i^2) \in \field[b_i^2,x] \subset \cM_X^+[x]
\end{align}
satisfy the initial conditions and recursion A) in Section \ref{sec:statement}. Hence we get $w_m(x)=w_m(x,q_i^2)$. Similarly,
\begin{align}
\label{eq:ivm-def}
  v_m(x) =  (b_i/2)^m H_m(\frac{x}{b_i};q_i^{-2}) \in \field[b_i^2,x] \subset \cM_X^+[x]
\end{align}
satisfy the recursion A) in Section \ref{sec:statement} with $q_i^{2m}$ replaced by $q_i^{-2m}$. Hence we get $v_m(x)=w_m(x,q_i^{-2})$.
Finally, for $i\in I\setminus X$ and $j\in I$ the rescaled bivariate continuous $q$-Hermite polynomials 
\begin{align}\label{eq:ijwmn-def}
  w_{m,n}(x,y) = (b_i/2)^{m+n} H_{m,n}(\frac{x}{b_i},\frac{y}{b_i};q_i^2,q_i^{a_{ij}}) \in \field[b_i^2,x,y] \subset \cM_X^+[x,y]
\end{align}
satisfy the initial conditions and recursion B) in Section \ref{sec:statement}. Hence we get $w_{m,n}(x,y)=w_{m,n}(x,y;q_i^2,q_i^{a_{ij}})$.
By Lemma \ref{lem:wmn} we have
\begin{align}\label{eq:ijw0n}
  w_{0,n}(x,y)= w_n(y)
\end{align}
for all $n\in \N$. All through Section \ref{sec:GenRel} the notations $w_m(x)$, $v_m(x)$ and $w_{m,n}(x,y)$ will refer to the above special instances of the polynomials investigated in Section \ref{sec:orth-polyn}. In particular, we may freely use the results of Section \ref{sec:orth-polyn}. To keep notation short we will suppress the dependence on $i$ and $j$.

We will evaluate the polynomials $w_m(x)$, $v_m(x)$ and $w_{m,n}(x,y)$ on elements in the algebra $(\cA,\ast)$. To distinguish the different algebra structures on $\cA$ we introduce some notation. For any $u\in \cA$ and any $n\in \N$ we write
\begin{align*}
  u^{\ast n}= \underbrace{u\ast u\ast \dots \ast u}_{\mbox{$n$ factors}}.
\end{align*}  
For any polynomial $w(x)=\sum_{n}\lambda_n x^n\in \cA[x]$ and any $u\in \cA$ we write
\begin{align}\label{eq:wast-def}
  w(u)^\ast = \sum_n \lambda_n \ast u^{\ast n}.
\end{align}
For any polynomial $w(x,y)=\sum_{s,t}\lambda_{st} x^s y^t\in \cA[x,y]$ and any $a_1, a_2, a_3 \in \cA$ we write
\begin{align}\label{eq:insertion-def}
   a_3 \curvearrowright w(a_1 \stackrel{\ast}{,} a_2) = \sum_{s,t} a_1^{\ast s} \ast a_3 \ast a_2^{\ast t}\ast \lambda_{st}.
\end{align}
\begin{rema}\label{rem:coeff-right}
  Note that we write the coefficients $\lambda_{st}$ on the right side in \eqref{eq:insertion-def}. This will not be relevant in Section \ref{sec:caseI}  where we will consider expressions of the form $F_j \curvearrowright w_{m,n}(F_i\stackrel{\ast}{,}F_i)$ for $i,j\in I\setminus X$. Indeed, the coefficients of the rescaled bivariate continuous $q$-Hermite polynomials $w_{m,n}(x,y)$ only involve powers of $b_i^2$ which commutes with $F_k$ for all $k\in I\setminus X$.
  
  However, in Section \ref{sec:jinX} we will consider the insertion operation
  $F_j \curvearrowright w_{m,n}(F_i\stackrel{\ast}{,}F_i)$ for $i\in I\setminus X$ and $j\in X$. In this case $b_i^2$ does not commute with $F_j$ and it will turn out crucial to have the coefficients $\lambda_{st}$ on the right side in \eqref{eq:insertion-def}.
\end{rema}
\subsection{The powers of $F_i$ for $\tau(i)=i$}
Fix $i\in I\setminus X$ with $\tau(i)=i$. The following proposition expresses $F_i^n$ for $n\in \N$ in terms of the star product $\ast$ on $\cA$. Recall \eqref{eq:iwm-def}, \eqref{eq:wast-def} and the recursion \eqref{eq:wm-recursion}.
\begin{prop}\label{prop:Fim}
  Let $i\in I\setminus X$ with $\tau(i)=i$. The relation
  \begin{align}\label{eq:Fim}
     F_i^m = w_m(F_i)^\ast
  \end{align}
  holds in $\cA$ for any $m\in \N$.
\end{prop}
\begin{proof}
  The relation \eqref{eq:Fim} holds for $m=0$ and $m=1$. We proceed by induction on $m$. By Lemmas \ref{lem:Fiastu} and \ref{lem:pLiXFn} we have
  \begin{align*}
    F_i\ast F_i^m&\stackrel{\phantom{\eqref{eq:pLiXFn}}}{=} F_i^{m+1} - c_i \frac{q^{(\alpha_i,w_X(\alpha_i))}}{q_i-q_i^{-1}} K_{w_X(\alpha_i)-\alpha_i} \partial_{i,X}^L(F_i^{m})\\
    &\stackrel{\eqref{eq:pLiXFn}}{=} F_i^{m+1} - \frac{c_i q_i^2}{q_i-q_i^{-1}} (m)_{q_i^2} Z_i F_i^{m-1}\\
    &\stackrel{\phantom{\eqref{eq:pLiXFn}}}{=} F_i^{m+1} + \frac{b_i^2}{4} (1-q_i^{2m}) F_i^{m-1}.
  \end{align*}  
  Hence by induction hypothesis and \eqref{eq:wm-recursion} we get
  \begin{align*}
     F_i^{m+1} = F_i\ast w_m(F_i)^\ast - \frac{b_i^2}{4} (1-q_i^{2m}) w_{m-1}(F_i)^\ast = w_{m+1}(F_i)^\ast
  \end{align*}
  which completes the induction step.
\end{proof}  

\subsection{The quantum Serre relation for $\tau(i)=i\neq j$ where $i,j\in I\setminus X$}\label{sec:caseI}
We now want to rewrite the quantum Serre relation $S_{ij}(F_i,F_j)=0$ for $\tau(i)=i$ with $i,j\in I\setminus X$ in terms of the  star product on the algebra $\cA$. Recall the rescaled bivariate $q$-Hermite polynomials \eqref{eq:ijwmn-def}, the insertion operator defined by \eqref{eq:insertion-def} and the recursion \eqref{eq:wmn-recursion}.
\begin{prop}\label{prop:FimFjFin}
  Let $i,j\in I\setminus X$ with $\tau(i)=i\neq j$. Then the relation
  \begin{align}\label{eq:FimFjFin}
    F_i^m F_j F_i^n = F_j \curvearrowright w_{m,n}(F_i\stackrel{\ast}{,}F_i)
  \end{align}
  holds in $\cA$ for any $m,n\in \N$.
\end{prop}
\begin{proof}
  We prove Equation \eqref{eq:FimFjFin} by induction on $m$. For $m=0$ the relation holds by Proposition \ref{prop:Fim} and Equation \eqref{eq:ijw0n}. 
Now fix $m,n\in \N$. Equations \eqref{eq:Fiastu}, \eqref{eq:partialLX} and \eqref{eq:pLiXFn} imply that
  \begin{align}
    F_i\ast F_i^mF_jF_i^n &=  F_i^{m+1} F_j F_i^n - \frac{c_i q^{(\alpha_i,w_X(\alpha_i))}}{q_i-q_i^{-1}} K_{w_X(\alpha_i)-\alpha_i} \partial^L_{i,X}(F_i^mF_jF_i^n)\nonumber\\
    &=  F_i^{m+1} F_j F_i^n - \frac{c_i q^{(\alpha_i,\alpha_i)}}{q_i-q_i^{-1}} (m)_{q_i^2} Z_i F_i^{m-1}F_j F_i^n \nonumber\\
    &\qquad \qquad \qquad  -  \frac{c_i q^{(\alpha_i,(m+1)\alpha_i+\alpha_j)}}{q_i-q_i^{-1}} (n)_{q_i^2} Z_i F_i^{m}F_j F_i^{n-1} \nonumber\\
    &=  F_i^{m+1} F_j F_i^n + \frac{b_i^2}{4} (1-q_i^{2m}) F_i^{m-1}F_j F_i^n \label{eq:Fiast}\\
    &\qquad \qquad \qquad\qquad   + \frac{b_i^2}{4}   q_i^{2m+a_{ij}}(1-q_i^{2n}) F_i^{m}F_j F_i^{n-1} \nonumber
  \end{align}
  Hence by induction hypothesis and \eqref{eq:wmn-recursion} for $r=q^{a_{ij}}$ we get
  \begin{align*}
    F_i^{m+1}F_j F_i^n &= F_i\ast \big( F_j \curvearrowright w_{m,n}(F_i\stackrel{\ast}{,}F_i)\big) - \frac{b_i^2}{4}(1{-}q_i^{2m})F_j \curvearrowright w_{m-1,n}(F_i\stackrel{\ast}{,}F_i)\\
    & \qquad \qquad \qquad -\frac{b_i^2}{4} q_i^{2m+a_{ij}}(1{-}q_i^{2n})F_j \curvearrowright w_{m,n-1}(F_i\stackrel{\ast}{,}F_i)\\
    &= F_j\curvearrowright \Big((x\cdot w_{m,n} - \frac{b_i^2}{4}(1{-}q_i^{2m}) w_{m-1,n}\\
    &\qquad \qquad \qquad\qquad\qquad   -\frac{b_i^2}{4} q_i^{2m+a_{ij}}(1{-}q_i^{2n}) w_{m,n-1} \Big)(F_i\stackrel{\ast}{,}F_i)\\
    &=F_j \curvearrowright w_{m+1,n}(F_i\stackrel{\ast}{,}F_i)
  \end{align*}
  which completes the induction step.
\end{proof}
With the above proposition we are able to express the quantum Serre relation $S_{ij}(F_i,F_j)=0$ in $\cA$ in terms of the star product. 
Recall \eqref{eq:iwm-def}--\eqref{eq:ijwmn-def}. 
\begin{thm} \label{thm:rel-I}
  Let $i,j\in I\setminus X$ with $\tau(i)=i\neq j$. Then the relation
  \begin{align*}
    \sum_{n=0}^{1-a_{ij}}(-1)^n \begin{bmatrix} 1-a_{ij}\\ n \end{bmatrix}_{q_i}
    F_j \curvearrowright w_{1-a_{ij}-n,n}(F_i \stackrel{\ast}{,}F_i) =0
  \end{align*}
  holds in the algebra $(\cA,\ast)$. This relation can be rewritten as
  \begin{align*}
    \sum_{n=0}^{1-a_{ij}}(-1)^n \begin{bmatrix} 1-a_{ij}\\ n \end{bmatrix}_{q_i}
    w_{1-a_{ij}-n}(F_i)^\ast \ast F_j \ast v_{n}(F_i)^\ast =0.
  \end{align*}
\end{thm}  
\begin{proof}
  The first statement is a direct consequence of the quantum Serre relation $S_{ij}(F_i,F_j)=0$ and Proposition \ref{prop:FimFjFin}. The second statement then follows from Proposition \ref{prop:Serre-bi-uni} with $a=a_{ij}$.
\end{proof}  
In view of the isomorphism $\psi:\cB_\bc\rightarrow (\cA,\ast)$, Theorem \ref{thm:rel-I} proves case (I) of Theorem \ref{thm:intro}.
\subsection{A recursive formula in the case $\tau(i)=i\neq j$ where $i\in I\setminus X$ and $j\in X$}\label{sec:jinX}
Assume that $i\in I\setminus X$ with $\tau(i)=i$ and $j\in X$. Recall from \eqref{eq:cZi} that in this case we write
\begin{align}
  &\cZ_i=c_iq_i Z_i=c_iq_i \partial_i^R(T_{w_X}(E_i)),\label{eq:Zi-cal}
\end{align}
Moreover, we set
\begin{align}
   d_{ij}&=\partial_j^R(\cZ_i)K_j, \qquad \widetilde{d_{ij}}=K_j^{-1}\partial_j^L(\cZ_i)\label{eq:dij}
\end{align}
so that Equation \eqref{eq:FxxF} now reads
\begin{align}\label{eq:FjZi}
  [F_j,\cZ_i] = \frac{\widetilde{d_{ij}}-d_{ij}}{q_j-q_j^{-1}}.
\end{align}
Recall from Section \ref{sec:anti-iso} that it is convenient to simultaneously consider the parameters $\bc=(c_i)_{i\in I\setminus X}\in \cC$ and $\bc'=(c_i')_{i\in I\setminus X}\in \cC$ related by \eqref{eq:c'-def}. We write $\cZ_i'$, $d'_{ij}$, $\widetilde{d'_{ij}}$ to denote the elements \eqref{eq:Zi-cal} and \eqref{eq:dij} corresponding to the parameter family $\bc'$. It follows from \cite[Lemma 2.9]{a-BalaKolb15} and \cite[Theorem 4.1]{a-BW18p} that
  \begin{align}\label{eq:cZibar}
     \overline{\cZ_i}=\overline{c_i} q_i (-1)^{2\alpha_i(\rho_X^\vee)}q^{(\alpha_i,\Theta(\alpha_i)-2\rho_X)} \partial_i^R(T_{w_X}(E_i))=\cZ'_i.
  \end{align}
  Moreover, we have $d_{ij}\in K_j\cM_X^+$ and $\widetilde{d_{ij}}\in K_j^{-1}\cM_X^+$. Hence, applying the bar involution to Equation \eqref{eq:FjZi} gives us
  \begin{align*}
    \overline{\widetilde{d_{ij}}}=d'_{ij}, \qquad \overline{d_{ij}}=\widetilde{d'_{ij}}.
  \end{align*}
These relations will be used in the proof of Theorem \ref{thm:caseII}.
  
The following lemma provides a formula similar to \eqref{eq:Fiast} in the present setting.
\begin{lem}\label{lem:Fiast-j}
  Let $i\in I\setminus X$ with $\tau(i)=i$ and $j\in X$. Then the relation
  \begin{align}
      F_i\ast &F_i^mF_j F_i^n
  = F_i^{m+1} F_j F_i^n \label{eq:Fiast-j}\\
   &\qquad + (1{-}q_i^{2m})  F_i^{m-1} F_j F_i^n \frac{b_i^2}{4} 
    + q_i^{2 m + a_{ij}} (1{-}q_i^{2n}) F_i^m F_j F_i^{n-1}\frac{b_i^2}{4} \nonumber\\
    &+ \frac{q_i^{(m-1)a_{ij}} (1{-}q_i^{2m})}{\lambda_{ij}}  d_{ij} F_i^{m+n-1}
    -\frac{q_i^{-(m-1)a_{ij}} (1{-}q_i^{2(m+n)})}{\lambda_{ij}}  \widetilde{d_{ij}} F_i^{m+n-1}\nonumber
    \end{align}
holds in $\cA$ for all $m,n\in \N$ with $b_i^2=4 \cZ_i/(q_i-q_i^{-1})^2$ and $\lambda_{ij}=(q_i-q_i^{-1})^2(q_j-q_j^{-1})$.
\end{lem}
\begin{proof}
 For all $m,n\in \N$ we have
  \begin{align}\label{eq:adr-FjFin}
  F_i^m F_j F_i^n=  q^{n(\alpha_j,\alpha_i)} F_i^{m+n} F_j - q^{n(\alpha_j,\alpha_i)}  F_i^m \ad_r(F_j)(F_i^n)
\end{align}
and hence Equation \eqref{eq:Fiastu} implies
\begin{align}
  F_i\ast F_i^mF_j F_i^n =& F_i^{m+1} F_j F_i^n - \frac{c_i q^{(\alpha_i,w_X(\alpha_i)+n\alpha_j)}}{q_i-q_i^{-1}} K_{w_X(\alpha_i)-\alpha_i} \partial_{i,X}^L(F_i^{m+n}) F_j\\
  &+  \frac{c_i q^{(\alpha_i,w_X(\alpha_i)+n\alpha_j)}}{q_i-q_i^{-1}} K_{w_X(\alpha_i)-\alpha_i}  \partial_{i,X}^L\big(F_i^m \ad_r(F_j)(F_i^n)\big).\nonumber
\end{align}
In view of the skew-derivation property \eqref{eq:partialLX}, Equations \eqref{eq:pLiXFn}, \eqref{eq:partialLiX-adrFjFin} allow us to rewrite the above as
\begin{align*}
   F_i\ast F_i^mF_j F_i^n
  =& F_i^{m+1} F_j F_i^n - \frac{c_i q^{(\alpha_i,\alpha_i+n\alpha_j)}}{q_i-q_i^{-1}} (m+n)_{q_i^2} Z_i F_i^{m+n-1} F_j\\
  &+ \frac{c_i q^{(\alpha_i,\alpha_i+n\alpha_j)}}{q_i-q_i^{-1}} (m)_{q_i^2} Z_i F_i^{m-1} \ad_r(F_j)(F_i^n)\\
  &+\frac{c_i q^{(m+1)(\alpha_i,\alpha_i+\alpha_j)-(\alpha_j,w_X(\alpha_i)-\alpha_i-\alpha_j)}}{(q_i-q_i^{-1})(q_j-q_j^{-1})} (n)_{q_i^2} K_j \partial_j^R(Z_i) F_i^{m+n-1}\\
  &+ \frac{c_i q^{(\alpha_i,(m+1)\alpha_i+n\alpha_j)}}{q_i-q_i^{-1}} (n)_{q_i^2} Z_i F_i^m \ad_r(F_j)(F_i^{n-1}).
\end{align*}
Using the relation $(m+n)_{q_i^2}=(m)_{q_i^2} + q_i^{2m}(n)_{q_i^2}$ and Equation \eqref{eq:adr-FjFin} for the third and fifth term, we obtain
\begin{align*}
    &F_i\ast F_i^mF_j F_i^n
    = F_i^{m+1} F_j F_i^n - \frac{c_i q_i^2}{q_i-q_i^{-1}}\Big( (m)_{q_i^2} Z_i  F_i^{m-1} F_j F_i^n + \\
    &\qquad + q_i^{2 m + a_{ij}} (n)_{q_i^2} Z_i F_i^m F_j F_i^{n-1}\Big)
     +\frac{c_i q_i^{(2m+2)-(n-2)a_{ij}}}{(q_i-q_i^{-1})(q_j-q_j^{-1})} (n)_{q_i^2}  F_i^{m+n-1} \partial_j^R(Z_i) K_j.
\end{align*}
Using the notation \eqref{eq:Zi-cal}, \eqref{eq:dij} and the commutation relation \eqref{eq:FjZi} one transforms the above equation into Equation \eqref{eq:Fiast-j}.
\end{proof}
Recall the insertion operation defined by \eqref{eq:insertion-def} and the rescaled bivariate continuous $q$-Hermite polynomials $w_{m,n}(x,y)$
defined by \eqref{eq:ijwmn-def}.
The recursion \eqref{eq:Fiast-j} implies that there exist polynomials $\rho_{m,n}(x), \sigma_{m,n}(x)\in \cM_X^+[x]$ such that
\begin{align}\label{eq:ansatz}
  F_i^m F_j F_i^n = F_j\curvearrowright w_{m,n}(F_i\stackrel{\ast}{,}F_i) + d_{ij} \rho_{m,n}(F_i)^\ast + \widetilde{d_{ij}} \sigma_{m,n}(F_i)^\ast.
\end{align}
For $m=0$ we have $F_jF_i^n=F_j\curvearrowright w_n(F_i)^\ast$ and hence $\rho_{0,n}(x)=\sigma_{0,n}(x)=0$ for all $n\in \N$. We can translate the recursion \eqref{eq:Fiast-j} into recursive formulas for $\rho_{m,n}(x)$ and $\sigma_{m,n}(x)$.
\begin{lem}\label{lem:rho-sigma-recursion}
  With the ansatz \eqref{eq:ansatz} the recursion \eqref{eq:Fiast-j} is equivalent to the recursions
  \begin{align}
    q_i^{a_{ij}} x&\rho_{m,n}(x)= \rho_{m+1,n}(x) +(1{-}q_i^{2m})\rho_{m-1,n}(x) \frac{b_i^2}{4}\label{eq:rhomn-recursion}\\
    &+ (1{-}q_i^{2n})q_i^{2m+a_{ij}} \rho_{m,n-1}(x) \frac{b_i^2}{4}
    + \frac{q_i^{(m-1)a_{ij}}(1{-}q_i^{2m})}{\lambda_{ij}} w_{m+n-1}(x),\nonumber\\
    q_i^{-a_{ij}} x&\sigma_{m,n}(x)= \sigma_{m+1,n}(x)   +(1{-}q_i^{2m}) \sigma_{m-1,n}(x) \frac{b_i^2}{4}\nonumber\\
    &+ (1{-}q_i^{2n})q_i^{2m+a_{ij}} \sigma_{m,n-1}(x) \frac{b_i^2}{4}
    - \frac{q_i^{-(m-1)a_{ij}}(1{-}q_i^{2(m+n)})}{\lambda_{ij}} w_{m+n-1}(x) \nonumber
  \end{align}
  for all $m,n\in \N$.
\end{lem}  
From now on we focus on the polynomials $\rho_{m,n}$ only. In Section \ref{sec:PN} we will determine the Serre combination of the polynomials $\rho_{m,n}$. The Serre combination of the polynomials $\sigma_{m,n}$ can then be obtained with the help of the isomorphism $\Phi$ from Section \ref{sec:anti-iso}.

We produce an unscaled versions of the recursion \eqref{eq:rhomn-recursion}. Define polynomials $U_{m,n}(x;q,r)\in \field[x]$ recursively by $U_{0,n}(x;q,r)=0$ for all $n\in \N$ and
\begin{align}
  2x U_{m,n}(x)= U_{m+1,n}(x) +& (1-q^m) r^{-1} U_{m-1,n}(x) + q^m(1-q^n)U_{m,n-1}(x) \label{eq:U-recursion}\\
  &+(1-q^m) H_{m+n-1}(x). \nonumber 
\end{align}
The following statement is an immediate consequence of Equation \eqref{eq:rhomn-recursion}.
\begin{lem}\label{lem:rhoU}
  The relation
  \begin{align*}
    \rho_{m,n}(x)=\frac{q_i^{(m-2)a_{ij}}}{\lambda_{ij}} (b_i/2)^{m+n-2} U_{m,n}(\frac{x}{b_i};q_i^2,q_i^{2a_{ij}})
  \end{align*}  
  holds for all $m,n\in \N$ with $b_i^2=4\cZ_i/(q_i-q_i^{-1})^2$.
\end{lem}  

\subsection{A generating function approach for $\tau(i)=i\neq j$, $i\in I\setminus X$, $j\in X$}
\label{sec:gf}
Define a generating function
\begin{align*}
  \psi\left(\begin{matrix} x \\s,t\end{matrix}\,;\,q\right)=\sum_{m,n\ge 0} \frac{H_{m+n}(x;q)}{(q;q)_m(q;q)_n}s^mt^n.
\end{align*}
Using the recursions \eqref{eq:Hm-recursion} and \eqref{eq:Hmn-recursion} and induction over $m$ one obtains
\begin{align*}
  H_{m,n}(x,x;q,1)=H_{m+n}(x;q).
\end{align*}
Hence \cite[Theorem 2.7]{a-CKY20} implies that
\begin{align*}
  \psi\left(\begin{matrix} x \\s,t\end{matrix}\,;\,q\right) = \frac{(st;q)_\infty}{|(se^{i\theta},te^{i\theta};q)_\infty|^2}
\end{align*}
for $x=\cos(\theta)$. Following \cite[Section 3.1]{a-CKY20} we have
\begin{align*}
  \psi\left(\begin{matrix} x \\qs,t\end{matrix}\,;\,q\right) = \frac{1-2xs+s^2}{1-ts} \psi\left(\begin{matrix} x \\s,t\end{matrix}\,;\,q\right).
\end{align*}
In terms of $x_{1,2}=x\pm \sqrt{x^2-1}$ defined in \eqref{eq:xa12}, this can be rewritten as
\begin{align}\label{eq:psiqst}
  \psi\left(\begin{matrix} x \\qs,t\end{matrix}\,;\,q\right) = \frac{(1-x_1s)(1-x_2s)}{1-ts} \psi\left(\begin{matrix} x \\s,t\end{matrix}\,;\,q\right).
\end{align}
Now recall the polynomials $U_{m,n}(x;q,r)$ defined by the recursion \eqref{eq:U-recursion} and consider the generating function
\begin{align}
\label{eq:phiU}
  \phi\left(\begin{matrix} x \\s,t\end{matrix}\,;\,q\right) = \sum_{m,n\ge 0} \frac{U_{m,n}(x;q,r)}{(q;q)_m(q;q)_n}s^mt^n. 
\end{align}
The initial condition $U_{0,n}(x;q,r)=0$ implies that
\begin{align}\label{eq:phi0t}
   \phi\left(\begin{matrix} x \\0,t\end{matrix}\,;\,q\right)=0.
\end{align}
Finally, recall the definition \eqref{eq:def-eta} of the function $\eta\left(\begin{matrix} x \\s \end{matrix}\,;\,q\right)$ which is analytic at $s=0$.
  \begin{lem}  
  \label{phi-psi}
    The relation
\begin{align}\label{eq:phi-psi}
\phi\left(\begin{matrix} x \\s,t\end{matrix}\,;\,q\right)=-s^2 \eta\left(\begin{matrix} x \\s \end{matrix}\,;\,q\right)  \psi\left(\begin{matrix} x \\s,t\end{matrix}\,;\,q\right)
\end{align}
    holds as an identity of formal power series in $s$ and $t$ with coefficients in $\field[x]$.
  \end{lem}  
\begin{proof}
The recursion \eqref{eq:U-recursion} implies that
\begin{align*}
  2x  \phi\left(\begin{matrix} x \\s,t\end{matrix}\,;\,q\right)=& \frac{1}{s}\left( \phi\left(\begin{matrix} x \\s,t\end{matrix}\,;\,q\right) -  \phi\left(\begin{matrix} x \\qs,t\end{matrix}\,;\,q\right)\right)\\
        &\qquad + s r^{-1}  \phi\left(\begin{matrix} x \\s,t\end{matrix}\,;\,q\right) + t  \phi\left(\begin{matrix} x \\qs,t\end{matrix}\,;\,q\right) +s  \psi\left(\begin{matrix} x \\s,t\end{matrix}\,;\, q\right)
\end{align*}
and hence
\begin{align*}
  (1-ts) \phi\left(\begin{matrix} x \\qs,t\end{matrix}\,;\,q\right)=(1-2xs + r^{-1}s^2)  \phi\left(\begin{matrix} x \\s,t\end{matrix}\,;\,q\right) + s^2  \psi\left(\begin{matrix} x \\s,t\end{matrix}\,;\,q\right).
\end{align*}
The above relation can be rewritten as
\begin{align*}
   \phi\left(\begin{matrix} x \\s,t\end{matrix}\,;\,q\right)= \frac{1-ts}{(1-a_1s)(1-a_2s)}  \phi\left(\begin{matrix} x \\qs,t\end{matrix}\,;\,q\right) - \frac{s^2}{(1-a_1s)(1-a_2s)}  \psi\left(\begin{matrix} x \\s,t\end{matrix}\,;\,q\right)
\end{align*}
where $a_{1,2}=x\pm \sqrt{x^2-r^{-1}}$ satisfy the relation $1-2xs+r^{-1}s^2=(1-a_1s)(1-a_2s)$, cf. \eqref{eq:xa12}.
By induction over $n$ the above formula together with Equation \eqref{eq:psiqst} gives
\begin{align*}
  \phi\left(\begin{matrix} x \\s,t\end{matrix}\,;\,q\right)=& \frac{(ts;q)_n}{(a_1s;q)_n(a_2s;q)_n}  \phi\left(\begin{matrix} x \\q^ns,t\end{matrix}\,;\,q\right)\\
      &\qquad\qquad - s^2 \sum_{k=0}^{n-1}\frac{q^{2k}(x_1s;q)_k(x_2s;q)_k}{(a_1s;q)_{k+1}(a_2s;q)_{k+1}}   \psi\left(\begin{matrix} x \\s,t\end{matrix}\,;\,q\right).
\end{align*}
Now we argue analytically for $q\in \C$ with $|q|<1$. In the limit $n\to \infty$ the first term in the above expressions vanishes by \eqref{eq:phi0t} and hence we get the desired formula \eqref{eq:phi-psi}.
\end{proof}
\subsection{The quantum Serre combination of the polynomials $\rho_{m,n}$}
\label{sec:PN}
To simplify notation set $N=1-a_{ij}$ and $q=q_i$. By Equation \eqref{eq:ansatz} and Lemma \ref{lem:rhoU} we need to determine the following polynomial
\begin{align}
  \sum_{m=0}^N(-1)^m\begin{bmatrix}N\\m\end{bmatrix}_{q}\rho_{N-m,m}(x)
    =&(-1)^N \sum_{m=0}^N(-1)^m\begin{bmatrix}N\\m\end{bmatrix}_{q}\rho_{m,N-m}(x) \nonumber\\
    =&(-1)^N \frac{q^{2(N-1)}}{\lambda_{ij}}(b_i/2)^{N-2}P_N(x/b_i;q) \label{eq:rho-Serre-combi}
\end{align}
where
\begin{equation}
\label{eq:P}
  P_N(x;q)=\sum_{m=0}^N (-1)^m\begin{bmatrix}N\\m\end{bmatrix}_q q^{(1-N)m} U_{m,N-m}(x;q^2,q^{2(1-N)}).
\end{equation}
Up to an overall factor, the polynomial $P_N(x;q)$ is a deformed Chebyshev polynomial of the second kind as defined in Section \ref{sec:Chebyshev}.
\begin{lem}\label{lem:PN}
For any $N\in \N$ we have 
\begin{align}\label{eq:PN}
P_N(x;q)=(-1)^{N-1} q^{(1-N)N}(q^2;q^2)_{N-1} \Dc_{N-2}(x;q^2, q^{2(1-N)}).
\end{align}
\end{lem}  
\begin{proof}
  By Proposition \ref{prop:Du-genfn}, Lemma \ref{phi-psi} and Equation \eqref{eq:phiU} we have
  \begin{align}
    P_N(x;q)&=\sum_{m=0}^N(-1)^{m-1}\begin{bmatrix}N\\m\end{bmatrix}_q q^{(1-N)m}(q^2;q^2)_m \cdot\nonumber\\
       & \qquad \qquad \qquad \qquad\qquad  \cdot\sum_{k=2}^m \frac{\Dc_{k-2}(x;q^2, q^{2(1-N)})}{1-q^{2k}}\, \frac{H_{N-k}(x;q^2)}{(q^2,q^2)_{m-k}}\nonumber\\
    &=\sum_{k=2}^N \frac{\omega_{N,k}(q)}{1-q^{2k}}\, \Dc_{k-2}(x;q^2, q^{2(1-N)}) H_{N-k}(x;q^2) 
    \label{eq:Pomegab}
  \end{align}
  where
  \begin{align*}
    \omega_{N,k}(q)=\sum_{m=k}^N (-1)^{m-1}\begin{bmatrix}N\\m\end{bmatrix}_q \frac{(q^2;q^2)_m}{(q^2;q^2)_{m-k}}q^{(1-N)m}.
  \end{align*}
  Using the relation $(q^2;q^2)_m=(-1)^m q^{m(m+1)/2}(q-q^{-1})^m [m]^!_q$ and setting $\ell=m-k$ we obtain
  \begin{align*}
   \omega_{N,k}(q)=- q^{k(k+1)/2}q^{(1-N)k}(q-q^{-1})^k \frac{[N]^!_q}{[N-k]^!_q} \sum_{\ell=0}^{N-k}(-1)^\ell \begin{bmatrix}N-k\\ \ell\end{bmatrix}_q q^{\ell(k-N+1)}
  \end{align*}
  By \cite[1.3.4]{b-Lusztig94} we obtain
  \begin{align*}
    \omega_{N,k}(q)&=\begin{cases}
                     0 & \mbox{if $k\neq N$,}\\
                     (-1)^{N-1} (q^2;q^2)_N q^{(1-N)N}& \mbox{if $k=N$.}
                   \end{cases}
  \end{align*}
  Inserting the above in \eqref{eq:Pomegab} we obtain the desired formula.
\end{proof}  
Inserting \eqref{eq:PN} into Equation \eqref{eq:rho-Serre-combi} we obtain
\begin{align}
   \sum_{m=0}^N(-1)^m&\begin{bmatrix}N\\m\end{bmatrix}_{q}\rho_{N-m,m}(x) \nonumber\\
    &=- \frac{q^{(N-2)(1-N)}}{\lambda_{ij}}\left(\frac{b_i}{2}\right)^{N-2} (q^2;q^2)_{N-1} \Dc_{N-2}(x/b_i;q^2, q^{2(1-N)}). \label{eq:rho-Serre-combi-final}
\end{align}  
\subsection{The quantum Serre relation for $\tau(i)=i$ where $i\in I\setminus X$ and $j\in X$}\label{sec:caseII}
We are now in a position to write down the deformed quantum Serre relation \eqref{eq:S=C-intro} in the case $i\in I\setminus X$, $\tau(i)=i$ and $j\in X$.
Recall the antilinear algebra isomorphism $\Phi:\cB_\bc\rightarrow \cB_{\bc'}$ from Theorem \ref{thm:bar}. Using the parameters $\bc$ and $\bc'$ we obtain two star products $\ast$ and $\ast'$ on $\cA$ such that $\cB_\bc\cong (\cA,\ast)$ and $\cB_{\bc'}\cong (\cA,\ast')$, respectively. Under these identifications we may consider $\Phi$ as an antilinear algebra isomorphism
\begin{align*}
  \Phi:(\cA,\ast) \rightarrow (\cA,\ast')
\end{align*}  
Also recall the elements $b_i^2=4\cZ_i/(q_i-q_i^{-1})^2$ and write $(b_i^2)'=4\cZ_i'/(q_i-q_i^{-1})^2$. By \eqref{eq:cZibar} we have
\begin{align}\label{eq:Phibi}
  \Phi(b_i^2)=\overline{b_i^2}=(b_i')^2
\end{align}
With these notational preliminaries we are ready to prove our main result.
\begin{thm}\label{thm:caseII}
  Let $i\in I\setminus X$ with $\tau(i)=i$ and $j\in X$. Then the relation
  \begin{align*}
    \sum_{n=0}^{1-a_{ij}}(-1)^n \begin{bmatrix} 1-a_{ij}\\ n \end{bmatrix}_{q_i}
     F_j \curvearrowright w_{1-a_{ij}-n,n}(F_i\stackrel{\ast}{,}F_i)  + C + D =0
  \end{align*}
  holds in the algebra $(\cA,\ast)$ where 
  \begin{align}
    C&= -\frac{d_{ij}}{\lambda_{ij}} q_i^{-a_{ij}(a_{ij} +1)} (q_i^2; q_i^2)_{-a_{ij}} \left(\frac{b_i}{2}\right)^{-a_{ij}-1} \Dc_{-a_{ij}-1} (\frac{F_i}{b_i}; q_i^2, q_i^{2 a_{ij}})^\ast,\label{eq:C}\\
    D&= \frac{\widetilde{d_{ij}}}{\lambda_{ij}}q_i^{a_{ij}(a_{ij} +1)} (q_i^{-2}; q_i^{-2})_{-a_{ij}} \left(\frac{b_i}{2}\right)^{-a_{ij}-1} \Dc_{-a_{ij}-1} (\frac{F_i}{b_i}; q_i^{-2}, q_i^{-2 a_{ij}})^\ast. \label{eq:D}
  \end{align}
\end{thm}
\begin{proof}
By Equation \eqref{eq:ansatz} we have the relation
\begin{align}\label{eq:ABC}
0=\sum_{n=0}^{1-a_{ij}} (-1)^n \begin{bmatrix}1-a_{ij} \\ n \end{bmatrix}_{q_i}F_i^{1-a_{ij} - n} F_j F_i^n
 = A+C+D 
 \end{align}
in $\cA$ with
\begin{align}
  A&= \sum_{n=0}^{1-a_{ij}} (-1)^n \begin{bmatrix}1-a_{ij} \\ n \end{bmatrix}_{q_i} F_j \curvearrowright w_{1- a_{ij} -n, n}(F_i\stackrel{\ast}{,}F_i)\label{eq:A}
\\
C&=  d_{ij} \sum_{n=0}^{1-a_{ij}} (-1)^n \begin{bmatrix}1-a_{ij} \\ n \end{bmatrix}_{q_i} \rho_{1-a_{ij} - n, n}(F_i)^*,\nonumber
\\
D&=  \widetilde{d_{ij}} \sum_{n=0}^{1-a_{ij}} (-1)^n \begin{bmatrix}1-a_{ij} \\ n \end{bmatrix}_{q_i} \sigma_{1-a_{ij} - n, n}(F_i)^*.\nonumber
\end{align}
By Equation \eqref{eq:rho-Serre-combi-final} we obtain
\begin{align*}
C&= -\frac{d_{ij}}{\lambda_{ij}} \left( \frac{b_i}{2} \right)^{-a_{ij}-1} q_i^{-a_{ij}(a_{ij} +1)} (q_i^2; q_i^2)_{-a_{ij}} \Dc_{-a_{ij}-1} (\frac{F_i}{b_i}; q_i^2, q_i^{2 a_{ij}})^\ast
\end{align*}
which proves Equation \eqref{eq:C}.
For the parameters $\bc'\in \cC$ defined by \eqref{eq:c'-def} we write Equation \eqref{eq:ABC} as
\begin{align*}
  A'+C'+D'=0.
\end{align*}  
Equation \eqref{eq:Phibi} and Corollary \ref{cor:def-qSerre-bar} imply that $\Phi(A)=A'$. Hence we obtain
\begin{align}\label{eq:C+D}
  \Phi(C) + \Phi(D) = C' + D'.
\end{align}
Relation \eqref{eq:dij} implies that
\begin{align*}
  C= K_j p_C(F_i)^\ast, \qquad  C'= K_j p_C'(F_i)^\ast, \\
  D= K_j^{-1} p_D(F_i)^\ast, \qquad D'= K_j^{-1} p_D'(F_i)^\ast
\end{align*}
for some polynomials $p_C(x), p_D(x), p_C'(x), p_D'(x)\in \cM_X^+[x]$. As $\Phi(K_j)=K_j^{-1}$, Equation \eqref{eq:C+D} hence implies that $\Phi(C)=D'$ and $\Phi(D)=C'$. This gives us
\begin{align*}
  D&=\Phi^{-1}(C')\\
   &=\frac{\widetilde{d_{ij}}}{\lambda_{ij}}\left( \frac{b_i}{2} \right)^{-a_{ij}-1} q_i^{a_{ij}(a_{ij} +1)} (q_i^{-2}; q_i^{-2})_{-a_{ij}} \Dc_{-a_{ij}-1} (\frac{F_i}{b_i}; q_i^{-2}, q_i^{-2 a_{ij}})^\ast
\end{align*}
which proves Equation \eqref{eq:D}.
\end{proof}
For $n\in \N$ set
\begin{align*}
  u_n(x;q_i^2,q_i^{2a_{ij}})=\left(\frac{b_i}{2}\right)^n C_n(\frac{x}{b_i};q_i^2,q_i^{2a_{ij}})\in \cM_X^+[x].
\end{align*}
The polynomials $u_n(x;q_i^2,q_i^{2a_{ij}})$ satisfy the initial conditions and recursion given in C) in Section \eqref{sec:statement}. With this notation we get
\begin{align*}
  C&= -\frac{d_{ij}}{\lambda_{ij}} q_i^{-a_{ij}(a_{ij} +1)} (q_i^2; q_i^2)_{-a_{ij}}  \Du_{-a_{ij}-1} (F_i; q_i^2, q_i^{2 a_{ij}})^\ast,\\
    D&= \frac{\widetilde{d_{ij}}}{\lambda_{ij}}q_i^{a_{ij}(a_{ij} +1)} (q_i^{-2}; q_i^{-2})_{-a_{ij}} \Du_{-a_{ij}-1} (F_i; q_i^{-2}, q_i^{-2 a_{ij}})^\ast.
\end{align*}
Inserting the above expressions for $C$ and $D$ and Equation \eqref{eq:dij} into the formula in Theorem \ref{thm:caseII}, we obtain case (II) of Theorem \ref{thm:intro} under the isomorphism $\psi:\cB_\bc\rightarrow (\cA,\ast)$.
\begin{rema}
  We can use the formula from Proposition \ref{prop:Serre-bi-uni} to rewrite the term $A$ given by \eqref{eq:A} as
  \begin{align*}
 A = \sum_{n=0}^{1-a_{ij}}(-1)^n \begin{bmatrix} 1-a_{ij}\\ n \end{bmatrix}_{q_i}
   w_{1-a_{ij}-n}(F_i)^\ast \ast F_j \ast v_{n}(F_i)^\ast.
  \end{align*}
  However, following Remark \ref{rem:coeff-right}, this formula needs to be interpreted such that any coefficients $b_i^k\in \cM_X^+$ coming from $w_{1-a_{ij}-n}(F_i)^\ast$ are moved to the right hand side of the factor $F_j$.
\end{rema}
\begin{eg}
  The formulas in the above Theorem allow us to write down explicit expressions for the deformed quantum Serre relations in the case $\tau(i)=i\in I\setminus X$ and $j\in X$. We obtain
  \begin{align*}
    \sum_{n=0}^{1-a_{ij}} (-1)^n \begin{bmatrix}1-a_{ij} \\ n \end{bmatrix}_{q_i} &F_j \curvearrowright w_{1- a_{ij} -n, n}(F_i\stackrel{\ast}{,}F_i)\\
    &=\begin{cases}
     0 & \mbox{if $a_{ij}=0$}\\
     -\displaystyle\frac{q_i d_{ij} + q_i^{-1}\widetilde{d_{ij}}}{(q_i-q_i^{-1})(q_j-q_j^{-1})}& \mbox{if $a_{ij}=-1$}\\
      [2]_{q_i}\displaystyle\frac{q_id_{ij} - q_i^{-1}\widetilde{d_{ij}}}{q_j-q_j^{-1}} F_i & \mbox{if $a_{ij}=-2$}.
      \end{cases}
  \end{align*}  
Under the identification $\cB_\bc\cong (\cA,\ast)$, the above relations coincide with the relations given in \cite[Theorem 7.8]{a-Kolb14} in the reformulation given in \cite[Theorem 3.9]{a-BalaKolb15}. Note that $\cZ_i$ in the present paper coincides with $-q_ic_i\cZ_i$ in \cite{a-BalaKolb15}. For $a_{ij}=-3$ we obtain the new relation
  \begin{align*}
    \sum_{n=0}^{1-a_{ij}} &(-1)^n \begin{bmatrix}1-a_{ij} \\ n \end{bmatrix}_{q_i} F_j \curvearrowright w_{1- a_{ij} -n, n}(F_i\stackrel{\ast}{,}F_i)\\
    &=  - [2]_{q_i} \frac{q_i^3-q_i^{-3}}{q_j-q_j^{-1}} (d_{ij}+\widetilde{d_{ij}}) F_i^{\ast 2} -
    \frac{q_i^3-q_i^{-3}}{(q_i-q_i^{-1})^2(q_j-q_j^{-1})} (q_i^3d_{ij}+q_i^{-3}\widetilde{d_{ij}}) \cZ_i.
  \end{align*}
 In the special case where $a_{ji}=-1$ with $q_i=q$ and $q_j=q^3$, the above formula reproduces the formula given in \cite[(4.19)]{a-RV20}. 
\end{eg}  
\subsection{The quantum Serre relation for $\tau(i)=j$ where $i,j\in I\setminus X$}\label{sec:caseIII}
All through this subsection we assume that $i,j\in I\setminus X$ with $\tau(i)=j\neq i$. In this case Equation \eqref{eq:Fiastu} implies that
\begin{align}\label{eq:Fiastu-tij}
  F_i\ast u = F_i u - c_i\frac{q^{(\alpha_i,w_X(\alpha_j))}}{q_i-q_i^{-1}} K_{w_X(\alpha_j)-\alpha_i} \partial_{j,X}^L(u)
\end{align}
for all $u\in \cR_X$. Hence for any $n\in \N$ we have $F_i^{\ast n}=F_i^n$ and
\begin{align}
  F_j\ast F_i^n &\stackrel{\phantom{\eqref{eq:Fiastu}}}{=}F_j F_i^n - c_j\frac{q^{(\alpha_j,w_X(\alpha_i))}}{q_i-q_i^{-1}} K_{w_X(\alpha_i)-\alpha_j} \partial_{i,X}^L(F_i^n)\nonumber \\
  &\stackrel{\eqref{eq:pLiXFn}}{=} F_j F_i^n - c_j\frac{q^{(\alpha_i,\alpha_j)}}{q_i-q_i^{-1}} (n)_{q_i^2}K_i K_j^{-1} Z_i F_i^{n-1}\nonumber\\
  &\stackrel{\phantom{\eqref{eq:Fiastu}}}{=}F_j F_i^n -c_j\frac{q_i^{na_{ij}-2n+2}}{q_i-q_i^{-1}} (n)_{q_i^2}  F_i^{n-1} K_i K_j^{-1}Z_i. \label{eq:Fj-ast-Fin}
\end{align}
Moreover, one shows by induction on $m$, using \eqref{eq:Fiastu-tij}, \eqref{eq:partialLX} and \eqref{eq:partialLXFi}, that
\begin{align}\label{eq:Fim-ast-FjFin}
  F_i^{\ast m} \ast F_j F_i^n = F_i^m F_j F_i^n - c_i \frac{q_i^{2n-(n-1)a_{ij}}}{q_i-q_i^{-1}} (m)_{q_i^2}F_i^{m+n-1} K_j K_i^{-1} Z_j.
\end{align}
Inserting \eqref{eq:Fj-ast-Fin} into \eqref{eq:Fim-ast-FjFin} we obtain
\begin{align}
  F_i^m F_j F_i^n =  F_i^{\ast m}\ast F_j \ast F_i^{\ast n} &+ c_i \frac{q_i^{2n-(n-1)a_{ij}}}{q_i-q_i^{-1}} (m)_{q_i^2} F_i^{m+n-1} K_j K_i^{-1} Z_j \label{eq:FimFjFin-tij}\\
  &+ c_j\frac{q_i^{na_{ij}-2n+2}}{q_i-q_i^{-1}} (n)_{q_i^2}  F_i^{m+n-1} K_i K_j^{-1}Z_i. \nonumber
\end{align}
Using the relations
\begin{align*}
   \sum_{n=0}^\ell (-1)^n \begin{bmatrix}\ell \\ n \end{bmatrix}_{q} q^{n(\ell + 1)} =(q^2;q^2)_\ell,  \qquad \sum_{n=0}^\ell (-1)^n \begin{bmatrix}\ell \\ n \end{bmatrix}_{q} q^{n(\ell-1)}=0
\end{align*}  
which hold for all $\ell \in \N$, one shows that
\begin{align}
   &\sum_{n=0}^{1-a_{ij}} (-1)^n \begin{bmatrix}1-a_{ij} \\ n \end{bmatrix}_{q_i} q_i^{n(2-a_{ij})}(1-a_{ij}-n)_{q_i^2} =-\frac{q_i^{-1}(q_i^{2};q_i^{2})_{1-a_{ij}}}{q_i-q_i^{-1}},\label{eq:sum1}\\
  &\sum_{n=0}^{1-a_{ij}} (-1)^n \begin{bmatrix}1-a_{ij} \\ n \end{bmatrix}_{q_i} q_i^{n(a_{ij}-2)}(n)_{q_i^2} =-\frac{q_i^{-1}(q_i^{-2};q_i^{-2})_{1-a_{ij}}}{q_i-q_i^{-1}}.\label{eq:sum2}
\end{align}
Using Equation \eqref{eq:FimFjFin-tij} and the formulas \eqref{eq:sum1}, \eqref{eq:sum2}, we can now rewrite the quantum Serre relation $S_{ij}(F_i,F_j)=0$ in terms of the star product $\ast$ on $\cA$. One obtains the following result.
\begin{thm}
  Let $i,j\in I\setminus X$ with $\tau(i)=j\neq i$ and set $N=1-a_{ij}$. Then the relation
\begin{align*}
  \sum_{n=0}^{N}&(-1)^n\begin{bmatrix}N \\ n \end{bmatrix}_{q_i} F_i^{\ast (N-n)}\ast F_j \ast F_i^{\ast n}\\
  &=\frac{c_i q_i^{-N} (q_i^{2};q_i^{2})_{N}}{(q_i-q_i^{-1})^2} F_i^{\ast(N-1)} K_j K_i^{-1}Z_j
    + \frac{c_jq_i (q_i^{-2};q_i^{-2})_{N}}{(q_i-q_i^{-1})^2}F_i^{\ast(N-1)} K_i K_j^{-1} Z_i
\end{align*}
holds in the algebra $(\cA,\ast)$.
\end{thm}
With the relation $\cZ_i=c_iq_iZ_j$, the above Theorem turns into case (III) of Theorem \ref{thm:intro} under the isomorphism $\psi:\cB_\bc\rightarrow (\cA,\ast)$.
\providecommand{\bysame}{\leavevmode\hbox to3em{\hrulefill}\thinspace}
\providecommand{\MR}{\relax\ifhmode\unskip\space\fi MR }
\providecommand{\MRhref}[2]{%
  \href{http://www.ams.org/mathscinet-getitem?mr=#1}{#2}
}
\providecommand{\href}[2]{#2}

\end{document}